\documentclass[12pt,a4paper]{amsart}
\usepackage[top=35mm, bottom=35mm, left=30mm, right=30mm]{geometry}
\usepackage{mathptmx,bm}
\usepackage{mathrsfs}
\usepackage{amscd}
\usepackage{changes}

\usepackage[all]{xy}
\usepackage{xcolor}
\usepackage[colorlinks=true,citecolor=blue]{hyperref}

\theoremstyle{plain}
\newtheorem{thmx}{Theorem}

\newtheorem{thm}{Theorem}[section]
\newtheorem{lem}[thm]{Lemma}
\newtheorem{prop}[thm]{Proposition}
\newtheorem{cor}[thm]{Corollary}

\theoremstyle{definition}
\newtheorem{defn}[thm]{Definition}

\newtheorem{rem}[thm]{Remark}
\newtheorem{ques}[thm]{Question}

\newcommand{\Z}{{\mathbb{Z}_+}}
\newcommand{\N}{\mathbb{N}}

\newcommand{\mZ}{\mathbb Z}
\newcommand{\U}{\mathcal U}

\DeclareMathOperator{\diam}{diam}
\DeclareMathOperator{\supp}{supp}

\begin{document}
\title{Positive topological entropy and $\Delta$-weakly mixing sets}
\author[W. Huang, J. Li, X.~Ye and X. Zhou]{Wen Huang, Jian  Li, Xiangdong Ye and Xiaoyao Zhou}
\address[W. Huang, X.~Ye, X. Zhou]{Wu Wen-Tsun Key Laboratory of Mathematics, USTC, Chinese Academy of Sciences and
School of Mathematics, University of Science and Technology of China,
Hefei, Anhui, 230026, P.R. China}
\address[J. Li] {Department of Mathematics, Shantou University, Shantou, Guangdong, 515063,P.R. China}
\email{wenh@mail.ustc.edu.cn; lijian09@mail.ustc.edu.cn; yexd@ustc.edu.cn; zhouxiaoyaodeyouxian@126.com}

\date{\today}

\begin{abstract}
The notion of $\Delta$-weakly mixing set is introduced,
which shares similar properties of weakly mixing sets.
It is shown that if a dynamical system has positive topological entropy,
then the collection of $\Delta$-weakly mixing sets is residual in the closure of
the collection of entropy sets in the hyperspace.
The existence of $\Delta$-weakly mixing sets in a topological dynamical system
admitting an ergodic invariant measure which is not measurable distal
is obtained.
Moreover, Our results generalize several well known results
and also answer several open questions.
\end{abstract}
\keywords{$\Delta$-weakly mixing, positive topological entropy, entropy sets,
measurable distal extensions, non-classical Li-Yorke chaos}
\subjclass[2010]{37B05, 37B40, 37A35}
\thanks{Huang is supported by NNSF of China (11225105, 11431012),
Li is supported by NNSF of China (11401362, 11471125),
Ye is supported by NNSF of China (11371339, 11431012)
and Zhou is supported by NNSF of China (11271191).}
\maketitle
\section{Introduction}

The main aim of the current paper is to understand how complicated a topological dynamical system
with positive topological entropy or with a non-distal ergodic invariant  measure could be.
In the process to do so, we strengthen several
well known results along the line, and answer affirmatively or negatively several open questions.
To explain the results that we obtain, we start with some definitions.

A (topological) dynamical system $(X,T)$ is a compact metric space $(X,\rho)$
with $T$ being a continuous map from $X$ to itself.
We say that a dynamical system $(X,T)$ is (topologically) transitive if
for every two non-empty open subsets $U$ and $V$ of $X$ there exists a positive integer $n$
such that $U\cap T^{-n}V\neq\emptyset$. It is (topologically) weakly mixing if
the product system $(X\times X,T\times T)$ is transitive. By the Furstenberg intersection Lemma, we know  that if
$(X,T)$ is weakly mixing, then for every $n\geq 2$, the $n$-th product system
$(X^n,T^{(n)}):=(X\times X\times\dotsb \times X, T\times T\times\dotsb \times T)$ ($n$-times) is transitive.

A dynamical system $(X,T)$ is $\Delta$-transitive if
for every $d\geq 2$ there exists a residual subset $X_0$ of $X$ such that
for every $x\in X_0$ the diagonal $d$-tuple $x^{(d)}=:(x,x,\dotsc,x)$  has a dense orbit under the action
$T\times T^2\times\dotsb\times T^d$.
There are two important classes of dynamical systems which are  $\Delta$-transitive.
Glasner~\cite{G94} proved that if a minimal system is weakly mixing
then it is $\Delta$-transitive.
Recently, the authors in \cite{KLOY14} showed that
if a dynamical system admits a weakly mixing invariant measure with full support
then it is $\Delta$-transitive.

In~\cite{M10}, Moothathu showed that $\Delta$-transitivity implies weak mixing,
but there exists some strongly mixing systems which are not $\Delta$-transitive.
Using a class of Furstenberg families introduced in \cite{CLL14a},
the authors in \cite{CLL14} characterized
the entering time sets of transitive points into open sets in $\Delta$-transitive systems.
We will give a more natural characterization of $\Delta$-transitive systems
by the generalized hitting time sets of open sets (see Proposition~\ref{prop:Delta-transitive}).

Intuitively, one can define that a dynamical system $(X,T)$ is $\Delta$-weakly mixing if
the product system $(X\times X,T\times T)$ is $\Delta$-transitive.
We show that this kind of $\Delta$-weakly mixing is in fact equivalent to $\Delta$-transitivity
(see Proposition~\ref{prop:Delta-transitive-weak-mixing}),
and then $\Delta$-transitivity shares similar properties of weakly mixing.

Blanchard and Huang \cite{BH08} introduced a local version of weak mixing, named it as weak mixing sets.
A closed subset $A$ of $X$ with at least two points
is weakly mixing if for any $k\in\N$, any non-empty open subsets $U_1,U_2,\dotsc,U_k$ and
$V_1,V_2,\dotsc,V_k$ of $X$ intersecting $A$ (that is $U_i\cap A\neq\emptyset$ and $V_i\cap A\neq\emptyset$)
there exists $m\in\N$ such that $U_i\cap T^{-m}V_i$ is a non-empty open set intersecting $A$
for each $1\leq i\leq k$.

Inspired by this we say that a closed subset $A$ of $X$ is $\Delta$-transitive if for every $d\geq 2$
there exists a residual subset $A_0$ of $A$ such that
for every $x\in A_0$, the orbit closure of the diagonal $d$-tuple $x^{(d)}$ under the action
$T\times T^2\times\dotsb\times T^d$ contains $A^d$.
A closed subset $A$ of $X$ with at least two points is $\Delta$-weakly mixing  if for every $n\geq 1$,
$A^n$ is $\Delta$-transitive in the $n$-th product system $(X^n,T^{(n)})$.
It is easy to see that $(X,T)$ is $\Delta$-weakly mixing if and only if
$X$ is a $\Delta$-weakly mixing set.
Note that when defining the local version, a subset $A$ is $\Delta$-weakly mixing is
not equivalent to the $\Delta$-transitivity of $A$ (see Remark \ref{rem:two-points-delta-transitve-sets}).

The first part of this paper studies the properties of $\Delta$-weakly mixing subsets
and $\Delta$-weakly mixing systems.  A nice characterization of weak mixing was obtained in \cite{XY92},
and was extended to weakly mixing sets in \cite{BH08}.
We show that $\Delta$-weakly mixing sets share a similar characterization as weakly mixing sets.

\begin{thmx}\label{thm:Delta-weakly-mixing-set}
Let $(X,T)$ be a dynamical system. Then a closed subset $E$ of $X$ is $\Delta$-weakly mixing if and only if
there exists a countable Cantor sets $C_1\subset C_2\subset \dotsb$ of $E$
such that $C=\bigcup_{k=1}^\infty C_k$
is dense in $E$ and
\begin{enumerate}
  \item for any $d\in\N$, any subset $A$ of $C$, and any continuous functions $g_j: A\to E$ for $j=1,2,\dotsc,d$
  there exists an increasing sequence $\{q_k\}_{k=1}^\infty$ of positive integers such that
  \[\lim_{i\to\infty}T^{j\cdot q_i}x=g_j(x)\]
  for every $x\in A$ and $j=1,2,\dotsc,d$;
  \item for any $d\in \N$, $k\in\N$, any closed subset $B$ of $C_k$,
  and continuous functions $h_j: B\to E$ for $j=1,2,\dotsc,d$
  there exists an increasing sequence $\{q_k\}_{k=1}^\infty$  of positive integers such that
  \[\lim_{k\to\infty}T^{j\cdot q_k}x=h_j(x)\]
  uniformly on $x\in B$ and $j=1,2,\dotsc,d$.
\end{enumerate}
\end{thmx}

To get a topological analogue  of Kolmogorov systems in ergodic theory,
Blanchard initiated  the study of local entropy entropy and  introduced the
notion of entropy pairs~\cite{B92,B93}.
Later, Huang and Ye \cite{HY06} introduced the notion of entropy $n$-tuples ($n\geq 2$)
in both the topological and measure theoretical settings.
In order to find where the entropy is concentrated,
the notion of an entropy set was introduced in \cite{DYZ} and  \cite{BH08}.
Moreover, the authors in \cite{BH08} proved that there are many weakly mixing
sets in dynamical system with  positive topological entropy involving the notion of entropy sets.
More precisely, let $WM(X, T)$ be the collection of weakly mixing sets of $(X,T)$ and
$H(X, T)$ be the closure of the collection of entropy sets in the hyperspace.
They showed that in a dynamical system $(X,T)$ with positive entropy,
$H(X, T )\cap WM(X, T )$ is a dense $G_\delta$ subset of $H(X,T)$.
This result was further generalized in \cite{L15}.
Let $\Delta$-$WM(X,T)$ be the collection of $\Delta$-weakly mixing sets in $(X,T)$.
In the second part of the paper we will strengthen the above result in \cite{BH08} by showing

\begin{thmx}\label{thm:positive-entopy-Delta-WM}
If a dynamical system $(X,T)$ has positive topological entropy,
then $H(X, T )\cap \Delta\text{-}WM(X, T )$ is a dense $G_\delta$ subset of $H(X,T)$.
\end{thmx}

In fact, by results in \cite{BGKM02} and the method in~\cite{BH08},
we know that if a dynamical system admits an ergodic invariant measure
which is not measurable distal then there are also many weakly mixing sets.
We strengthen this result by showing the existence of $\Delta$-weakly mixing sets
in the same setting. We note that unlike the method
in the previous results we need to use results from the proof of multiple ergodic theorem by Furstenberg
in~\cite{F81}, not only the structure theorem of an ergodic system developed
by Furstenberg~\cite{F81} and Zimmer~\cite{Z76,Z76-2}.

\begin{thmx}\label{thm:measureable-WM}
Let $(X,T)$ be a dynamical system.
If there exists an ergodic invariant measure on $\mu$ such that
$(X,\mathcal{B},\mu,T)$ is not measurable distal,
then there exist ``many'' $\Delta$-weakly mixing sets in $(X,T)$.
\end{thmx}

We remark that Theorems B and C can be generalized to $\mZ^d$-actions easily.
Since in the proofs we need to use Szemeredi's Theorem,
which is not clear for  countable discrete amenable group actions,
we do not know whether Theorems B and C are still valid
for countable discrete amenable group actions.

In the final section we will discuss the non-classical Li-Yorke chaos,
 which generalizes several well known results in \cite{AGHSY}
and also answer  several open questions in \cite{M12}
(see Questions \ref{q-1}, \ref{q-2} and \ref{q-3} stated in our terminology).

\section{Preliminaries}
In this section, we present some basic notations, definitions and results.
\subsection{Density of subsets of non-negative integers}
Let $\N$ and $\mathbb{Z}_+$ denote the collection of positive integers and non-negative integers respectively.
For a subset $F$ of $\mathbb{Z}_+$, the {\it upper density} of $F$ is defined by
\[\overline{D}(F)=\limsup_{n\to\infty} \frac{|F\cap\{0,1,\dotsc,n-1\}|}{n},\]
where $|\cdot|$ is the number of elements of a finite set.
Similarly, the {\it lower density} of $F$ is defined by
\[\underline{D}(F)=\liminf_{n\to\infty} \frac{|F\cap\{0,1,\dotsc,n-1\}|}{n}.\]
We  may say $F$ has {\it density} $D(F)$ if
$\underline{D}(A)=\overline{D}(A)$, in which case $D(A)$ is equal to this common value.
We will need the following famous Szemeredi's Theorem~\cite{S75},
see \cite{F81} for a ergodic theory proof by Furstenberg.
\begin{thm}[Szemeredi's Theorem]
If $F\subset\Z$ has positive upper density, then it contains arithmetic progressions of arbitrary finite
length, that is for every $d\geq 1$ there exist $m,n\geq 1$ such that $\{m,m+n,m+2n,\dotsc,m+dn\}\subset F$.
\end{thm}

A subset $F$ of $\N$ is {\it thick} it contains arbitrarily long blocks of consecutive integers, that is, for every
$N\geq 1$ there is $n\in\N$ such that $\{n,n+1,\dotsc,n+N\}\subset F$.
A collection $\mathcal{F}$ of subsets of $\N$ is called a {\it filter base} if
for any $F_1,F_2\in\mathcal{F}$ there exists a non-empty set $F\in\mathcal{F}$ such that $F\subset F_1\cap F_2$.

\subsection{Compact metric spaces and hyperspaces}
Let $(X,\rho)$ be a compact metric space.
We say that a non-empty subset $A$ of $X$ is {\it totally disconnected}
if there only connected subsets of $A$ are singletons.
By a {\it perfect set} we mean a non-empty closed set without isolated point,
by a {\it Cantor set} we mean a compact, perfect and totally disconnected set,
and by a {\it Mycielski set} we mean a set which can be expressed as a union of countably many Cantor sets.
For convenience we restate here a version of Mycielski's theorem (\cite[Theorem 1]{M64}) which we shall use.

\begin{thm}[Mycielski Theorem]
Let X be a perfect compact metric space.
If $R$ is a dense $G_\delta$ subset of $X\times X$, then there exists a dense Mycielski subset $K$ of $X$
such that for any two distinct points $x,y\in K$, the pair $(x,y)\in R$.
\end{thm}

We will consider the following hyperspace associated to be the space $X$:
$2^X$, the collection  of all non-empty closed subsets of $X$.
We may endow  $2^X$ with
the {\it Hausdorff metric} $\rho_H$ defined by
\[\rho_H(A,B)=\max\Bigl\{\max_{x\in A}\min_{y\in B} \rho(x,y),\ \max_{y\in B}\min_{x\in A}
\rho(x,y)\Bigr\}\ \text{for}\ A,B\in 2^X.\]
The following family
\[
\{\langle U_1, \dotsc, U_n\rangle\colon U_1, \dotsc, U_n \textrm{ are non-empty open subsets of } X,\ n\in \mathbb{N}\}
\]
forms a basis for a topology of $2^X$ called the {\it Vietoris topology}, where
\[
\langle S_1,\dotsc, S_n\rangle:= \Bigl\{A\in 2^X\colon A\subset
\bigcup_{i=1}^n S_i \textrm{ and } A\cap S_i\neq \emptyset \textrm{ for each }i=1,\dotsc, n\Bigr\}
\]
is defined for arbitrary non-empty subsets $S_1,\dotsc, S_n\subset X$.
It is not hard to see that the Hausdorff topology (the topology induced by the Hausdorff metric $\rho_H$)
and the Vietoris topology for $2^X$ coincide. For more details on hyperspaces we refer the reader to~\cite{N78}.

A subset $Q$ of $2^X$ is called {\it hereditary} if $2^A\subset Q$ for every set $A\in Q$.
We will need the following lemma,
which is a consequence of the Kuratowski-Mycielski Theorem (see~\cite[Theorem 5.10]{A04}).

\begin{lem}\label{lem:resudial-relation}
Suppose that $X$ is a perfect compact space.
If a hereditary subset $Q$ of $2^X$ is residual then
there exists a countable Cantor sets $C_1\subset C_2\subset \dotsb$ of $X$
such that $C_i\in Q$ for every $i\geq 1$ and $C=\bigcup_{i=1}^\infty C_i$ is dense in $X$.
\end{lem}

\subsection{Topological dynamics}
By a (topological) dynamical system, we mean a pair $(X,T)$
consisting of a compact metric space $(X,\rho)$ and a continuous map $T\colon X\to X$.
If $K \subset X$ is a non-empty closed subset
satisfying $T(K) \subset K$, then we say that $(K,T)$ is a {\it subsystem} of $(X,T)$ and
$(X,T)$ is {\it minimal} if it has no proper subsystems.
A point $x\in X$ is {\it minimal} if it is contained in some minimal subsystem of $(X,T)$.
A point $x\in X$ is a {\it fixed point} if $Tx=x$, or a {\it periodic point}
if there exists $n\in\N$ such that $T^n x=x$.

For two subsets $U$ and $V$ of $X$, we define the {\it hitting time set of $U$ and $V$} by
\[
N(U,V)=\{n\in\N\colon U\cap T^{-n}V\neq\emptyset\}.
\]
A dynamical system $(X,T)$ is {\it transitive} if for every two non-empty open subsets $U$ and $V$ of $X$,
the hitting time set $N(U,V)$ is not empty;
{\it weakly mixing} if the product system $(X\times X,T\times T)$ is transitive.
We have the following well known characterizations of weak mixing, see \cite[Theorem~1.11]{G03} for example.
\begin{prop} \label{prop:weak-mixing}
Let $(X,T)$ be a dynamical system. Then the following conditions are equivalent
\begin{enumerate}
  \item  $(X,T)$ is weakly mixing;
  \item the collection of hitting time sets,
  \[\{N(U,V)\colon U, V\text{ are non-empty open subsets of }X\},\]
  is a filter base;
  \item for non-empty open subsets $U$ and $V$ of $X$,
   the hitting time set $N(U,V)$ is thick;
\item for every $m\geq 2$, $(X^m,T^{(m)})$ is transitive, where $T^{(m)}=T\times T\times \dotsb\times T$ ($m$-times).
\end{enumerate}
\end{prop}
For a point $x\in X$ and a subset $U$ of $X$, we define the {\it entering time set of $x$ into $U$} by
\[
N(x,U)=\{n\in\N\colon T^nx\in U\}.
\]
A point $x\in X$ is {\it transitive} if the orbit of $x$, $\{T^nx\colon n\in\Z\}$, is dense in $X$.
Note that if $(X,T)$ is transitive then either $X$ is a periodic orbit or perfect.
If $X$ is perfect or $T$ is surjective, then $(X,T)$ is transitive if and only if
there exists a transitive point in $X$ if and only if
the collection of transitive points of $(X,T)$ forms a dense $G_\delta$ subset of $X$.

\section{\texorpdfstring{$\Delta$}{Delta}-transitivity and \texorpdfstring{$\Delta$}{Delta}-weakly mixing sets}
In this section, inspired by results in Blanchard and Huang \cite{BH08} and Moothathu \cite{M10},
we will define and study a local version of $\Delta$-transitive
and $\Delta$-weakly mixing properties
which are called $\Delta$-transitive and $\Delta$-weakly mixing subsets, respectively.

\subsection{\texorpdfstring{$\Delta$}{Delta}-transitivity}
Recall that a dynamical system $(X,T)$ is called {\it $\Delta$-transitive} if for every $d\geq 2$ there
is a residual subset $X_0$ of $X$ such that for every $x\in X_0$
the diagonal $d$-tuple $(x,x,\dotsc,x)$ is transitive under the action
$T\times T^2\times \dotsb\times T^d$, that is the orbit $\{(T^nx,T^{2n}x,\dotsc,T^{dn}x)\colon n\in \Z\}$ is dense in $X^{d}$.
Recently, the authors in \cite{CLL14} characterized  $\Delta$-transitive systems
by the entering time sets of transitive points into open sets.
To get a better understanding of $\Delta$-transitivity, we generalize the hitting time sets between open subsets.
For every  $d\geq2$ and subsets $U_1,U_2,\dotsc,U_d$ of $X$, we define the {\it hitting time set}
of $U_1,U_2,\dotsc,U_d$ by
\begin{align*}
N(U_1,U_2,\dotsc,U_d)=\biggl\{n\in\N\colon \bigcap\limits_{i=1}^dT^{-(i-1)n}U_i\neq\emptyset\biggr\}.
\end{align*}
We have the following characterization of $\Delta$-transitive systems.

\begin{prop} \label{prop:Delta-transitive}
Let $(X, T)$ be a dynamical system. Then the following conditions are equivalent:
\begin{enumerate}
\item $(X,T)$ is $\Delta$-transitive;
\item $T$ is surjective and for every $d\geq 1$, there exists a point $x\in X$ such that
the diagonal $d$-tuple $(x,x,\dotsc,x)$ is transitive under the action $T\times T^2\times \dotsb\times T^d$;
\item for every $d\geq2$ and non-empty open subsets $U_1,U_2,\dotsc U_d$ of $X$,
$N(U_1,U_2,\dotsc,U_d)\neq\emptyset$.
\end{enumerate}
\end{prop}
\begin{proof}
(1) $\Rightarrow$ (2). If $(X,T)$ is $\Delta$-transitive, then it is clear transitive and thus $T$ is surjective.

(2) $\Rightarrow$ (3).
Fix $d\geq 2$. There exists a point $x\in X$
such that the diagonal $d$-tuple $(x,x,\dotsc,x)$ is transitive under the action
$T\times T^2\times \dotsb\times T^d$.
Fix non-empty open subsets $U_1,U_2,\dotsc,U_d$ of $X$.
Since $T$ is surjective,
$T^{-1}U_1\times T^{-1}U_2\times \dotsb \times T^{-1}U_d$ is a non-empty open subset of $X^{d}$.
There exists $n\in \Z$ such that
\[(T^nx,T^{2n}x,\dotsc,T^{dn}x )\in T^{-1}U_1\times T^{-1}U_2\times \dotsb \times T^{-1}U_d.\]
Then
\[T^{n+1}x\in \bigcap_{i=1}^d T^{-(i-1)n}U_i,\]
that is $n+1\in N(U_1,U_2,\dotsc,U_d)$.

(3) $\Rightarrow$ (1). Fix a countable topology base $\{U_k\colon k\in\N\}$ of $X$.
For every $d\geq 2$, put
\[ X_d=\bigcap_{(k_1,\dotsc,k_d)\in\N^d} \bigcup_{n\in \N} \bigcap_{i=1}^d T^{-in}U_{k_i}.\]
For every non-empty open subset $V$ of $X$, one has $N(V, U_{k_1},U_{k_2},\dotsc,U_{k_d})\neq\emptyset$.
That is, there exists $n\in \N$ such that $V\cap \bigcap_{i=1}^d T^{-in}U_{k_i}\neq\emptyset$.
So $\bigcup_{n\in \N} \bigcap_{i=1}^d T^{-in}U_{k_i}$ is open and dense in $X$.
And then by the Baire Category Theorem, $X_d$ is a dense $G_\delta$ subset of $X$.
Let $X_0=\bigcap_{d=2}^\infty X_d$. Then $X_0$ is also a dense $G_\delta$ subset of $X$.
By the construction, it is easy to check that  for every $x\in X_0$ and  $d\geq 2$
the diagonal $d$-tuple $x^{(d)}=:(x,x,\dotsc,x)$ is transitive under the action
$T\times T^2\times\dotsb\times T^d$.
Therefore $(X,T)$ is $\Delta$-transitive.
\end{proof}

Recall that a dynamical system $(X,T)$ is {\it weakly mixing}
if the product system $(X\times X,T\times T)$ is transitive.
Intuitively, one can define the concept of $\Delta$-weak mixing as follows:
a dynamical system $(X,T)$ is {\it $\Delta$-weakly mixing}
if the product system $(X\times X,T\times T)$ is $\Delta$-transitive.
But the following result reveals that this kind  $\Delta$-weak mixing is in fact equivalent to $\Delta$-transitivity.
Moreover, $\Delta$-transitivity shares many similar properties of weak mixing
(comparing Proposition~\ref{prop:weak-mixing}).

\begin{prop} \label{prop:Delta-transitive-weak-mixing}
 Let $(X,T)$ be a dynamical system. Then the following conditions are equivalent:
 \begin{enumerate}
 \item $(X,T)$ is $\Delta$-transitive;
 \item the collection of hitting time sets,
 $\{N(U_1,U_2,\dotsc, U_d)\colon d\geq2 $ and $U_1,U_2,\dotsc,U_d$ are non-empty open subsets of $X\}$,
 is a filter base;
 \item for every $d\geq 2$ and non-empty open subsets $U_1,U_2,\dotsc,U_d$ of $X$,
 $N(U_1,U_2,\dotsc,U_d)$ is thick;
 \item for every $m\geq 2, (X^m,T^{(m)})$ is $\Delta$-transitive.
 \end{enumerate}
 \end{prop}

 \begin{proof}
(4)$\Rightarrow$(1) is obvious, and (3) $\Rightarrow$ (1) follows from Proposition~\ref{prop:Delta-transitive}.

(1)$\Rightarrow$ (2)
We first show that if $X$ is $\Delta$-transitive, then $(X,T)$ is weakly mixing.
This was first proved in \cite[Proposition 3]{M10}, we provide a proof here for completeness.
For any non-empty open subsets $U_1,U_2,U_3,U_4$ of $X$, by Proposition~\ref{prop:Delta-transitive}
we have $N(U_1,U_3,U_2,U_4)\neq\emptyset$, that is there exists $n\in \N$ such that
\begin{align*}
 U_1\cap T^{-n}U_3\cap T^{-2n}U_2\cap T^{-3n}U_4\neq\emptyset.
\end{align*}
Thus, $ U_1\cap T^{-n}U_3\neq\emptyset$ and $U_2\cap T^{-n}U_4\neq\emptyset$,
which means that $(U_1\times U_2)\cap(T\times T)^{-n} (U_3\times U_4)\neq\emptyset$
and then $(X,T)$ is weakly mixing.

Now fix $d\geq2$ and non-empty open subsets $U_1,\dotsc, U_d$ and $V_1,\dotsc,V_d$ of $X$.
Since $(X,T)$ is weakly mixing, by Proposition~\ref{prop:weak-mixing} there exists $m\in\N$ such that
$W_i=U_i\cap T^{-m}V_i\neq\emptyset$ for $i=1,2,\dotsc,d$.
It is easy to check that
\begin{align*}
 N(W_1,W_2,\dotsc,W_d)\subset  N(U_1,U_2,\dotsc,U_d)\cap  N(V_1,V_2,\dotsc,V_d).
\end{align*}
This means that the collection is a filter base.

(2) $\Rightarrow$ (3)
Fix $d\geq 2$ and non-empty open subsets $U_1,U_2,\dotsc,U_d$ of $X$.
For every $N\in\N$, one has
\[\bigcap_{i=0}^N N(U_1,T^{-i}U_2,\dotsc,T^{-(d-1)i}U_d)\neq\emptyset.\]
Pick an integer $n$ in this intersection.
Note that $n\in N(U_1,T^{-i}U_2,\dotsc,T^{-(d-1)i}U_d)$ means
\[U_1\cap T^{-(n+i)}U_2\cap \dotsb\cap T^{-(d-1)(n+i)}U_d\neq\emptyset,\]
that is
\[n+i \in N(U_1,U_2,\dotsc,U_d).\]
So $\{n,n+1,\dotsc,n+N\}\subset N(U_1,\dotsc,U_d)$  and then it is thick.

(2) $\Rightarrow$ (4)
Fix $m\geq2$. For every $d\geq2$ and non-empty open subsets $V_1,V_2,\dotsc, V_d $ of $X^m$,
there exist non-empty open subsets $\{U_{i,j}\}_{1\leq i\leq m,1\leq j\leq d}$ of $X$ such that
$\prod\limits_{i=1}^m U_{i,j}\subset V_j$ for $j=1,2,\dotsc,d$.
Since the collection of hitting time sets is a filter base, we have
\begin{align*}
N(V_1,V_2,\dotsc,V_d)
&\supset N\Bigl(\prod\limits_{i=1}^m U_{i,1},\prod\limits_{i=1}^m U_{i,2},\dotsb, \prod\limits_{i=1}^m U_{i,d}\Bigr)\\
&\supset\bigcap\limits_{i=1}^m N(  U_{i,1},  U_{i,2},\dotsc,  U_{i,d}) \neq\emptyset.
 \end{align*}
It follows from Proposition \ref{prop:Delta-transitive} that $(X^m,T^{(m)})$ is $\Delta$-transitive.
\end{proof}

\subsection{\texorpdfstring{$\Delta$}{Delta}-transitive subsets}
The authors in \cite{OZ12} introduced a local version of transitivity---transitive subsets.
Let $(X,T)$ be a dynamical system and $A$ be a closed subset of $X$.
We say that $A$ is {\it transitive} if for every two non-empty open subsets $U$ and $V$ of $X$
intersecting $A$ (that is $U\cap A\neq\emptyset$ and $V\cap A\neq\emptyset$), $N(U\cap A,V)\neq\emptyset$.
Similarly, we can define the local version of $\Delta$-transitivity.
We say that a closed subset $A$ of $X$ is {\it $\Delta$-transitive} if for every $d\geq2$ and
non-empty open subsets $U_1,U_2,\dotsc,U_d$ of $X$ intersecting $A$,
\[
N(U_1\cap A, U_2,\dotsc,U_d)\neq\emptyset.
\]

Similarly to Proposition~\ref{prop:Delta-transitive},
we have the following characterization of $\Delta$-transitive subsets which shows that the definition given here is
equivalent to the one given in the introduction.
Since the proof is almost the same, we omit it here.

\begin{prop}\label{prop:Delta-transitive-sets},
Let $(X,T)$ be a dynamical system and $A$ be a closed subset of $X$.
Then $A$ is $\Delta$-transitive if and only if
for every $d\geq 2$ there exists a residual subset $A_0$ of $A$ such that for every $x\in A_0$, the orbit closure
of $d$-tuple $(x,\dotsc,x)$ under the action $T\times T^2\times \dotsb\times T^d$ contains $A^d$.
\end{prop}

\begin{rem}
In fact in Proposition~\ref{prop:Delta-transitive-sets}, one can also show that
$A$ is $\Delta$-transitive if and only if  there exists a dense set $A_0$ of $A$ such that for every $d \in \N$ and every $x\in A_0$
the diagonal $d$-tuple $(x,x,\dotsc,x)$  has a dense orbit under the action
$T\times T^2\times\dotsb\times T^d$.

Unlike the case when $A=X$ in Proposition~\ref{prop:Delta-transitive},
we can not remove the density condition of $A_0$ when considering $\Delta$-transitive sets.
For example, in the full shift $(\{0,1\}^{\mathbb{Z}},\sigma)$,
pick a point $x\in \{0,1\}^{\mathbb{Z}}$ such that for every $d\geq 2$, the
the diagonal $d$-tuple $(x,x,\dotsc,x)$  has a dense orbit under the action
$\sigma\times \sigma^2\times\dotsb\times \sigma^d$.
But $\{x,0^\infty\}$ is not a $\Delta$-transitive set.
\end{rem}

\subsection{\texorpdfstring{$\Delta$}{Delta}-weakly mixing subsets}
The authors in \cite{BH08} introduced a local version of weak mixing---weakly mixing subsets.
Let $(X,T)$ be a dynamical system and $A$ be a closed subset of $X$ with at least two points.
We say that $A$ is {\it weakly mixing} if for every $n\geq 1$, $A^n$ is a
transitive subset in $(X^n,T^{(n)})$, that is  for every $n\geq 1$ and
non-empty open subsets $U_1,U_2,\dotsc,U_n$ and $V_1,V_2,\dotsc,V_n$ of $X$ intersecting $A$,
\[\bigcap_{i=1}^n N(U_i\cap A,V_i)\neq\emptyset.\]

Inspired by Proposition~\ref{prop:Delta-transitive-weak-mixing},
we introduce the concept of $\Delta$-weakly mixing sets as follows.
We say that $A$ is {\it $\Delta$-weakly mixing}
if for every $n\geq1$, $A^n$ is a $\Delta$-transitive subset in $(X^n, T^{(n)})$,
that is for every $n,d\geq2$ and non-empty open subsets $U_{i,j}$ of $X$ intersecting $A$
for $i=1,2,\dotsc,n$  and $j=1,2,\dotsc, d$,
\begin{align*}
\bigcap\limits_{i=1}^nN(U_{i,1}\cap A,U_{i,2},\dotsc, U_{i,d})\neq\emptyset.
\end{align*}
It is clear that every $\Delta$-weakly mixing set is weakly mixing.
By \cite[Proposition 4.2.]{BH08} every weakly mixing set is perfect,
then every $\Delta$-weakly mixing set is also perfect.
By Proposition~\ref{prop:Delta-transitive-weak-mixing},
we know that a dynamical system is $\Delta$-weakly mixing if and only if it is $\Delta$-transitive.
This is no longer true for $\Delta$-weakly mixing and $\Delta$-transitive sets, as the following
remark shows.
\begin{rem}\label{rem:two-points-delta-transitve-sets}
In the full shift $(\{0,1\}^{\mathbb{Z}},\sigma)$,
pick two distinct points $x_1,x_2\in \{0,1\}^{\mathbb{Z}}$ such that for every $d\geq 2$,
the diagonal $d$-tuple $(x_i,x_i,\dotsc,x_i)$  has a dense orbit under the action
$\sigma\times \sigma^2\times\dotsb\times \sigma^d$ for $i=1,2$.
Then $\{x_1,x_2\}$ is a $\Delta$-transitive set. As $\{x_1,x_2\}$ is not perfect, so
it is not a $\Delta$-weakly mixing set.
\end{rem}

Denote by $WM(X,T)$ the collection of all weakly mixing sets in $(X,T)$.
It is shown in \cite[Theorem 4.4]{BH08} that $WM(X,T)$ is a $G_\delta$ subset of  $2^X$.
Denote by $\Delta$-$WM(X,T)$ the collection of all $\Delta$-weakly mixing sets in $(X,T)$.
We will show that $\Delta$-$WM(X,T)$ is also a $G_\delta$ subset of $2^X$.
Our proof is motivated by the method in \cite[Theorem 4.4]{BH08},
but our proof is simpler since we avoid the discussion of the perfectness of subsets.

For a dynamical system $(X,T)$ and $m\in\N$,
we define a subset $A_m(X,T)$ of $2^X$ as follows:
a closed subset $E$ of $X$ with at least two points is in $A_m(X,T)$ if and only if
there are  $k\geq 2$ and open subsets $\{W_i\}_{i=1}^k$ of $X$ such that

\begin{enumerate}
  \item $\bigcup\limits_{i=1}^k W_i\supset E$;
  \item $\diam(W_i)<\frac{1}{m}$  for $i=1,2\dotsc,k$;
  \item $\bigcap\limits_{\alpha\in \{1,2,\dotsc,k\}^k}N(W_{\alpha(1)}
  \cap E,W_{\alpha(2)},\dotsc,W_{\alpha(k)})\neq\emptyset$.
\end{enumerate}


\begin{lem}
For every $m\geq 1$, $A_m(X,T)$ is open in $2^X$.
\end{lem}
\begin{proof}
Fix $E\in A_m(X,T)$. By the definition there exist $k\geq 2$ and open subsets $\{W_i\}_{i=1}^k$ of $X$
satisfying (1)-(3).
Then there exists $m\in\N$ such that for any $\alpha\in\{1,2,\dotsc,k\}^k$
there exists $x_\alpha\in X$ with
\begin{align*}
x_\alpha\in (W_{\alpha(1)}\cap E)\cap T^{-m} W_{\alpha(2)}\cap\dotsb\cap T^{-(k-1)m}W_{\alpha(k)}.
\end{align*}
Take sufficiently small $\epsilon_1>0$ such that for any $\alpha\in\{1,2,\dotsc,k\}^k$,
\begin{align*}
B(x_\alpha,\epsilon_1)\subset  W_{\alpha(1)} \cap T^{-m} W_{\alpha(2)}\cap\dotsb\cap T^{-(k-1)m}W_{\alpha(k)}.
\end{align*}
Since $E\subset\bigcup\limits_{i=1}^k W_i$,
there exists $\epsilon_2>0$ such that for any $F\in 2^X$ with $\rho_H(F,E)<\epsilon_2$,
 one has $F\subset\bigcup\limits_{i=1}^k W_i$ and $F$ also has at least two points.
Put $\epsilon=\min\{\epsilon_1,\epsilon_2\}$.
To show that $A_m(X,T)$ is an open subset of $2^X$,
it is sufficient to check that $\{G\in2^X \colon \rho_H(G,E)<\epsilon\}\subset A_m(X,T)$.
Let $G\in 2^X$ with $\rho_H(G,E)<\epsilon$.
Since $\epsilon<\epsilon_2$, $G\subset\bigcup\limits_{i=1}^kW_i$ and $G$ has at least two points.
For any $\alpha\in\{1,2,\dotsc,k\}^k,$ since $\epsilon\leq\epsilon_1$ and $x_\alpha\in E,$ one has
\begin{align*}
G\cap B(x_\alpha,\epsilon_1)\neq\emptyset,
\end{align*}
and then for any $\alpha\in\{1,2,\dotsc,k\}^k$
\begin{align*}
(W_{\alpha(1)}\cap G)\cap T^{-m} W_{\alpha(2)} \cap\dotsb\cap T^{-(k-1)m} W_{\alpha(k)}\supset B(x_\alpha,\epsilon_1)\cap G\neq\emptyset.
\end{align*}
Thus, $G\in A_m(X,T)$.
\end{proof}

\begin{lem}\label{lem:Delta-WM-G-delta}
Let $(X,T)$ be a dynamical system.
Then $\Delta$-$WM(X,T)=\bigcap_{m=1}^\infty A_m(X,T)$.
In particular,  $\Delta$-$WM(X,T)$ is a $G_\delta$-subset of $2^X$.
\end{lem}

\begin{proof}
Let $E$ be a $\Delta$-weakly mixing set.
Fix $m\geq 1$. Note that  $E$ is perfect,
then there are $k\geq 2$ and open subsets $\{W_i\}_{i=1}^k$ of $X$ such that
$\bigcup\limits_{i=1}^k W_i\supset E$,
$\diam(W_i)<\frac{1}{m}$ and $W_i\cap E\neq\emptyset$ for $i=1,2\dotsc,k$.
As $W_1,W_1,\dotsc,W_k$ are non-empty open subsets of $X$ intersecting $E$,
by the definition of $\Delta$-weakly mixing sets there exists $m\in\mathbb N$ such that
for any $\alpha\in\{1,2,\dotsc,k\}^k,$
\begin{align*}
(W_{\alpha(1)}\cap E)\cap T^{-m} (W_{\alpha(2)})\cap\dotsb\cap   T^{-(k-1)m} (W_{\alpha(k)})\neq\emptyset.
\end{align*}
So $E\in A_m(X,T)$ and then $E\in \bigcap\limits_{m=1}^\infty A_m(X,T)$.

Conversely, let $E\in \bigcap\limits_{m=1}^\infty A_m(X,T)$.
Fix $n,d\geq2$ and non-empty open subsets $U_{i,j}$ of $X$ intersecting $A$
for $i=1,2,\dotsc,n$ and $j=1,2,\dotsc, d$.
Pick a sufficiently large integer $m$ such that for every $i,j$
there exists $m_{i,j}\in\{1,\dotsc,k\}$ such that $W_{m_{i,j}}\subset U_{i,j}$,
where $\{W_l\}_{l=1}^k$ are  as in the definition of $A_m(X,T)$ for $E$.
By the definition of $A_m(X,T)$, one has
\begin{align*}
\bigcap\limits_{\alpha\in \{1,2,\dotsc,k\}^k}N(W_{\alpha(1)}\cap E,W_{\alpha(2)},\dotsb,W_{\alpha(k)})\neq\emptyset,
\end{align*}
In particular, one has
\begin{align*}
\bigcap\limits_{i=1}^nN(U_{i,1}\cap E, U_{i,2},\dotsc, U_{i,d})\neq\emptyset.
\end{align*}
This implies that $E$ is $\Delta$-weakly mixing and then ends the proof.
\end{proof}

Now, we turn to the proof of Theorem A, which characterizes the chaotic behavior of $\Delta$-weakly mixing sets.
There are several ways to prove it, here we follow some ideas in~\cite{O11}.

Let $(X,T)$ be a dynamical system and $E$ be a closed subset of $X$.
For $\varepsilon>0$ and $d\geq 1$,
we say that a subset $A$ of $X$ is {\it $(\varepsilon,d)$-spread in $E$}
if there exist $\delta\in (0,\varepsilon)$, $n\in\N$
and distinct points $z_1,z_2,\dotsc,z_n\in X$
such that $A\subset \bigcup_{i=1}^n B(z_i,\delta)$ and
for any map $g_j\colon \{z_1,z_2,\dotsc,z_n\}\to E$ for $j=1,2,\dotsc,d$,
there exists $k\in\N$ such that $1/k<\varepsilon$
and $T^{j\cdot k}(B(z_i,\delta))\subset B(g_j(z_i),\varepsilon)$ for $i=1,2,\dotsc,n$ and $j=1,2,\dotsc,d$.

Denote by $\mathcal{X}(\varepsilon,d,E)$ the collection of all closed sets $(\varepsilon,d)$-spread in $E$.
Note the $\mathcal{X}(\varepsilon,d,E)$ is hereditary, that is if $A$ is  $(\varepsilon,d)$-spread in $E$
and $B$ is a non-empty closed subset of $A$ then $B$ is also  $(\varepsilon,d)$-spread in $E$.
Put
\[\mathcal{X}(E)=\bigcap_{d=1}^\infty\bigcap_{k=1}^\infty \mathcal{X}(\tfrac{1}{k},d,E).\]

\begin{lem}\label{lem:Delta-weak-mixing-residual}
If $E$ is a $\Delta$-weakly mixing subset of $X$,
then $\mathcal{X}(E)\cap 2^E$ is a residual subset of $2^E$.
\end{lem}
\begin{proof}
Fix $\varepsilon>0$ and $d\in\N$.
We first show that $\mathcal{X}(\varepsilon,d,E)$ is open in $2^X$.
Let $A\in \mathcal{X}(\varepsilon,d,E)$.
Let $\delta>0$ and $z_1,z_2,\dotsc,z_n\in X$ as in the definition of set $(\varepsilon,d)$-spread in $E$.
For $i=1,2,\dotsc,n$, put $U_i=B(z_i,\delta)$.
It is easy to see that every closed subset $B\subset \bigcup_{i=1}^n U_i$ is also $(\varepsilon,d)$-spread in $E$.
In particular $A\in \langle U_1, \dotsc, U_n\rangle \subset \mathcal{X}(\varepsilon,d,E)$.
So $\mathcal{X}(\varepsilon,d,E)$ is open in $2^X$.

Now we show that $\mathcal{X}(\varepsilon,d,E)\cap 2^E$ is dense in $2^E$.
Fix non-empty open subsets $U_1,\dotsc,U_n$ of $X$ intersecting $E$.
We want to show that $\langle U_1, \dotsc, U_n\rangle\cap \mathcal{X}(\varepsilon,d,E)\cap 2^E\neq\emptyset$.
Since $E$ is compact, there exists a finite subset $\{y_1,y_2,\dotsc,y_m\}$ of $E$
such that $\bigcup_{i=1}^m B(y_i,\frac{\varepsilon}{2})\supset E$.
For convenience,  denote $V_i=B(y_i,\frac{\varepsilon}{2})$ for $i=1,2,\dotsc,m$.
Since $E$ is $\Delta$-weakly mixing, $E$ is perfect, so we can assume that $m\geq \frac{1}{\varepsilon}$.
We arrange the collection of $n$-tuples on the set $\{1,2,\dotsc,m\}^d$ as the finite sequence
$\{(\alpha_{1,1}, \dotsb,\alpha_{1,n}),
(\alpha_{2,1}, \dotsb,\alpha_{2,n}),\dotsc,(\alpha_{\ell,1}, \dotsb,\alpha_{\ell,n})\}$.

For $(\alpha_{1,1}, \dotsb,\alpha_{1,n})$, as $E$ is $\Delta$-weakly mixing there exists $k_1>m$
such that
\[(U_i\cap E)\cap T^{-k_1}V_{\alpha_{1,i}(1)}\cap T^{-2k_1}V_{\alpha_{1,i}(2)}\cap\dotsb\cap T^{-dk_1}V_{\alpha_{1,i}(d)}\neq\emptyset \]
for $i=1,2,\dotsc,n$.
Choose a non-empty subset $W_i^1$ of $U_i$ intersecting $E$
such that $T^{j\cdot k_1}W_i^1\subset V_{\alpha_{1,i}(j)}$ for $j=1,2,\dotsc,d$.
For $(\alpha_{2,1}, \dotsb,\alpha_{2,n})$,
there exist  $k_2>k_1$ and non-empty subsets $W_i^2$ of $W_i^1$ intersecting $E$
such that $T^{j\cdot k_2}W_i^2\subset V_{\alpha_{2,i}(j)}$
for $i=1,2,\dotsc,n$ and $j=1,2,\dotsc,d$.
We apply this process inductively obtaining positive integers $k_1,k_2,\dotsc,k_\ell$ and
non-empty open subsets $W_i^1,W_i^2,\dotsc,W_i^\ell$ intersecting $E$ such that
\[W_i^\ell\subset W_i^{\ell-1}\subset \dotsb\subset W_i^1\subset U_i\]
and
\[T^{j\cdot k_r}W_i^r\subset V_{\alpha_{r,i}(j)}\]
for $r=1,2,\dotsc,\ell$, $i=1,2,\dotsc,n$ and $j=1,2,\dotsc,d$.
For $i=1,2,\dotsc,n$, pick $z_i\in W_i^\ell\cap E$.
Since $E$ is perfect, we can assume that those $z_i$'s are distinct.
It is clear that $\{z_1,\dotsc,z_n\}\in \langle U_1, \dotsc, U_n\rangle\cap 2^E$.
Choose $\delta\in(0,\varepsilon)$ such that $B(z_i,\delta)\subset W_i^\ell$ for $i=1,2,\dotsc,n$.
For any map $g_j\colon \{z_1,z_2,\dotsc,z_n\}\to E$ for $j=1,2,\dotsc,d$,
there exists $(\alpha_{\ell_0,1},\dotsc,\alpha_{\ell_0,n})$ such that
$V_{\alpha_{\ell_0,i}(j)}\subset B(g_j(z_i),\varepsilon)$ for $i=1,2,\dotsc,n$ and $j=1,2,\dotsc,d$.
Let $k=k_{\ell_0}$. One has
\[T^{j\cdot k}(B(z_i,\delta))\subset T^{j\cdot k_{\ell_0}}W_i^{\ell_0}
\subset V_{\alpha_{\ell_0,i}(j)} \subset B(g_j(z_i),\varepsilon)\]
 for $i=1,2,\dotsc,n$ and $j=1,2,\dotsc,d$.
This implies $\{z_1,\dotsc,z_n\}$ is $(\varepsilon,d)$-spread in $E$ and
then $\mathcal{X}(\varepsilon,d,E)\cap 2^E$ is dense in $2^E$.

Finally,  by the above discussions and the definition of $\mathcal{X}(E)$,
one has $\mathcal{X}(E)\cap 2^E$ is a residual subset of $2^E$.
\end{proof}

\begin{lem}\label{lem:cal-X}
If $C\in \mathcal{X}(E)$, then
 for any closed subset $B$ of $C$,  $d\in \N$
  and continuous functions $h_j\colon B\to E$ for $j=1,2,\dotsc,d$
  there exists an increasing sequence $\{q_k\}_{k=1}^\infty$  of positive integers such that
  \[\lim_{k\to\infty}T^{j\cdot q_k}x=h_j(x)\]
  uniformly on $x\in B$ and $j=1,2,\dotsc,d$.
\end{lem}

\begin{proof}
 Fix a closed subset $B$ of $C$,  $d\in \N$ and continuous functions $h_j\colon B\to E$ for $j=1,2,\dotsc,d$.
For any $k\in\N$, by the uniform continuity of $h_j$
there exists $\eta>0$ such that if $\rho(y,z)<\eta$ then $\rho(h_j(y),h_j(z))<\frac 1{2k}$ for $j=1,2,\dotsc,d$.
Choose $k'>2k$ with $\frac{1}{k'}<\eta$.
Let $z_{k,1},\dotsc,z_{k,n_{k}}$ and $0<\delta_{k}<\tfrac 1 {k'}$ be as in the definition of
$B$ which is $(\tfrac{1}{k'},d)$-spread in $E$.
There exists $q_k>k'$ such that
\[T^{j\cdot q_k}(B(z_{k,i},\delta_k))\subset B(h_j(z_{k,i}),\tfrac{1}{k'})\]
for $i=1,2,\dotsc,n_{k}$ and $j =1,2,\dotsc,d$.
Going to a subsequence, we may assume that $\{q_k\}_{k=1}^\infty$ is increasing.
We show that the sequence $\{q_k\}$ is as required.

For any $x\in B$, there exists $z_{k,n^x_{k}}$ such that
$\rho(x,z_{k,n^x_{k}})<\delta_k<1/k'<\eta$.
Then
\begin{align*}
  \rho(T^{j\cdot q_k}x,h_j(x))&\leq \rho(T^{j\cdot q_k}x, h_j(z_{k,n^x_{k}}))+\rho(h_j(z_{k,n^x_{k}}),h_j(x))\\
  &<\frac{1}{k'}+\frac{1}{2k}<\frac{1}{k}
\end{align*}
for $j=1,2,\dotsc,d$.
\end{proof}

\begin{lem}\label{lem:cal-X-2}
If $C_1\subset C_2\subset \dotsb$ be an increasing sequence of sets in $\mathcal{X}(E)$, then
for any subset $A$ of $C:=\bigcup_{i=1}^\infty C_i$, $d\in\N$ and continuous functions $g_j\colon A\to E$ for $j=1,2,\dotsc,d$
  there exists an increasing sequence $\{q_k\}_{k=1}^\infty$ of positive integers such that
  \[\lim_{k\to\infty}T^{j\cdot q_k}x=g_j(x)\]
  for every $x\in A$ and $j=1,2,\dotsc,d$.
\end{lem}

\begin{proof}
Fix a subset $A$ of $C$, $d\in\N$ and continuous functions $g_j\colon A\to E$ for $j=1,2,\dotsc,d$.
For $k\geq 1$, let $A_k=A\cap C_k$.
Since $\mathcal{X}(E)$ is hereditary,
the closure $\overline{A_k}$ of $A_k$ is also in $\mathcal{X}(E)$ for all $k\geq 1$.

For any $k\in\N$,  $\overline{A_k}$ is $(\tfrac{1}{k},d)$-spread in $E$.
Let $z_{k,1},\dotsc,z_{k,n_k}$ and $0<\delta_k<\tfrac 1 k$ be as in the definition
of $\overline{A_k}$ which is $(\tfrac{1}{k},d)$-spread in $E$.
There exists $q_k>k$ such that
\[T^{j\cdot q_k}(B(z_{k,i},\delta_k))\subset B(g_j(z_{k,i}),\tfrac{1}{k})\]
for $i=1,2,\dotsc,n_k$ and $j =1,2,\dotsc,d$.
Going to a subsequence, we may assume that $\{q_k\}_{k=1}^\infty$ is increasing.
We show that the sequence $\{q_k\}$ is as required.

Fix any $x\in A$. There exists $K_1\in\N$ such that $x\in A_k$ for all $k\geq K_1$.
By the continuity of $g_j$, for every $\varepsilon>0$ there exists $K_2\in \N$
with $K_2>\frac{2}{\varepsilon}$
such that if $\rho(x,y)<\frac{1}{K_2}$ then $\rho(g_j(x),g_j(y))<\frac{\varepsilon}{2}$ for $j=1,2,\dotsc,d$.
For every $k>\max\{K_1,K_2\}$, there exists $z_{k,n^x_{k}}$ such that
$\rho(x,z_{k,n^x_{k}})<\delta_k<1/k$.
Then
\begin{align*}
  \rho(T^{j\cdot q_k}x,g_j(x))&\leq \rho(T^{j\cdot q_k}x, g_j(z_{k,n^x_{k}}))+\rho(g_j(z_{k,n^x_{k}}),g_j(x))\\
  &<1/k+\varepsilon/2<\varepsilon
\end{align*}
for $j=1,2,\dotsc,d$.
\end{proof}

Now we give the proof of Theorem~\ref{thm:Delta-weakly-mixing-set}.
\begin{proof}[Proof of Theorem~\ref{thm:Delta-weakly-mixing-set}]
We first prove the necessity. Since $E$ is $\Delta$-weakly mixing,
by Lemma~\ref{lem:Delta-weak-mixing-residual} $\mathcal{X}(E)\cap 2^E$ is a residual subset of $2^E$.
Now by a consequence of the  Kuratowski-Mycielski Theorem (see Lemma~\ref{lem:resudial-relation}),
there exists a countable Cantor sets $C_1\subset C_2\subset \dotsb$ of $E$
such that $C_i\in \mathcal{X}(E)$ for every $i\geq 1$ and $C=\bigcup_{i=1}^\infty C_i$ is dense in $E$.
So the conclusion follows from Lemmas~\ref{lem:cal-X} and~\ref{lem:cal-X-2}.

Now we prove the sufficiency.
Let $C$ be the set satisfying the requirement.
Fix $n, d\geq2$ and non-empty open subsets $U_{i,j}$ of $X$ intersecting $E$
for $i=1,2,\dotsc,n$ and $j=1,2,\dotsc, d$.
It is clear that $E$ is perfect.
There exists pairwise distinct points $x_{ij}\in U_{i,j}\cap E$  for $i=1,2,\dotsc,n$ and $j=1,2,\dotsc, d$.
For $j=1,2,\dotsc,d-1$, define $g_j\colon A=\{x_{i1}\colon i=1,2,\dotsc,n\}\to E$ by
$g_j(x_{i1})=x_{i(j+1)}$ for $i=1,2,\dotsc,n$.
Choose a small enough $\varepsilon>0$ such that $B(x_{ij},\varepsilon)\subset U_{ij}$ for all $i,j$.
It is clear that $g_j$ are continuous, thus we can find $k\in\N$ such that
$\rho(T^{j\cdot k}x_{i1},g_j(x_{i1}))<\varepsilon$ for all $i,j$.
Then $T^{j\cdot k}x_{i1}\in U_{i(j+1)}$ for all $i,j$, which implies
\[k\in \bigcap\limits_{i=1}^nN(U_{i,1}\cap A,U_{i,2},\dotsc, U_{i,d})\neq\emptyset.\]
Therefore, $E$ is $\Delta$-weakly mixing.
\end{proof}

Let $(X,T)$ be a dynamical system.
Following \cite{KL07},
for a tuple $\bm{A}=(A_1, A_2, \dotsc,A_d)$ of subsets of $X$, we say that a non-empty subset $F\subset\Z$
is an {\it independence set} for $\bm{A}$ if for any non-empty finite subset $J\subset F$,
we have
\[\bigcap_{j\in J}T^{-j}A_{s(j)}\neq\emptyset\]
for any $s\in\{1,\dotsc,d\}^J$.
The notion of independence sets
was first presented in~\cite{HY06} under the name {\it interpolating set} (see also \cite{GW95}),
and in~\cite{H06} when defining {\it strong scrambled pairs}.
The authors in \cite{HLY12} systematically studied independence sets in topological and measurable dynamics.
In particular, they characterized  weakly mixing by the independent sets of open sets.

To get a characterization of weakly mixing sets,
the authors in \cite{LOZ15} introduced a local version of independence sets.
Let $A$ be a subset of $X$ and $U_1,U_2,\dotsc,U_n$ be open subsets of $X$ intersecting $A$.
We say that a non-empty subset $F\subset\Z$
is an {\it independence set for $(U_1,U_2,\dotsc,U_n)$ with respect to $A$},
if for every non-empty finite subset $J\subset F$, and $s\in\{1,2,\dotsc,n\}^J$,
\[\bigcap_{j\in J} T^{-j} U_{s(j)} \]
is a non-empty open subset of $X$ intersecting $A$.
We have the following characterization of $\Delta$-weakly mixing subset.

\begin{prop}\label{prop:Delta-WM-independent-set}
Let $(X,T)$ be a dynamical system and $A\subset X$ be a closed set with at least two points.
Then $A$  is $\Delta$-weakly mixing if and only if
for every $d\geq 2$ and non-empty open subsets $U_1,U_2,\dotsc,U_d$ of $X$ intersecting $A$,
there exists $n\in\N$ such that $\{0,n,2n,\dotsc,(d-1)n\}$ is an independent set for $(U_1,U_2,\dotsc,U_d)$
with respect to $A$.
\end{prop}
\begin{proof}
Necessity.
If $A$ is $\Delta$-weakly mixing,
then for any $d\geq2$ and non-empty open subsets $U_1,U_2,\dotsc,U_d$ of $X$ intersecting $A$,
by the definition of $\Delta$-weakly mixing sets we can choose
\begin{align*}
n\in \bigcap\limits_{s\in\{1,\dotsc,d\}^k}N(U_{s(1)}\cap A,U_{s(2)},\dotsc, U_{s(k)}).
\end{align*}
This implies that
$\{0,n,2n,\dotsc,(k-1)n\}$ is an independent set for $(U_1,U_2,\dotsc,U_d)$
with respect to $A$.

Sufficiency. For every $n,d\geq2$ and non-empty open subsets $U_{i,j}$ of $X$ intersecting $A$ for $i=1,2,\dotsc,n$ and $j=1,2,\dotsc,d,$
there exists $m\in\mathbb N$ such that $\{0,m,2m,\dotsc,ndm\}$ is an independence set
 of $ (U_{i,j})_{1\leq i\leq n,1\leq j\leq d}$ with respect to $A$. Then
\begin{align*}
m\in\bigcap\limits_{i=1}^nN(U_{i,1}\cap A,U_{i,2},\dotsb, U_{i,d})\neq\emptyset,
\end{align*}
which implies that $A$ is $\Delta$-weakly mixing.
\end{proof}

\section{Dynamical systems with positive topological entropy}
Let $(X,T)$ be a dynamical system.
An {\it open cover} is a family of open sets in $X$ whose union is $X$.
The {\it join} of  two open covers $\mathcal{U}$ and $\mathcal{V}$ of $X$, denoted by
 $\mathcal{U}\vee\mathcal{V}$, is the open cover $\{U\cap V\colon U\in\mathcal{U}$ and $V\in\mathcal{V}\}$.
For an open cover $\mathcal{U}$ of $X$,
define $N(\mathcal{U})$ as the minimum among the cardinals of the subcovers of $\mathcal{U}$.
The {\it topological entropy of $\mathcal{U}$ with respect to $T$} is
\[h(T,\mathcal{U})=\lim_{n\to\infty} \frac{1}{n}\log N\biggl(\bigvee_{i=1}^{n-1}T^{-i}\mathcal{U}\biggr),\]
and the {\it topological entropy} of $T$ is
\[h(T)=\sup h(T,\mathcal{U})\]
where the supremum ranges over all open covers of $X$.

Given a positive integer $n\geq 2$,
an $n$-tuple $(x_1,x_2,\dotsc,x_n)$ of points in $X$ is
a {\it topological entropy $n$-tuple} if at least two of the points $(x_1,x_2,\dotsc,x_n)$ are different and
if whenever $U_i$ are closed mutually disjoint neighborhoods of distinct points $x_i$,
the open cover $\{U_i^c\colon i=1,2,\dotsc,n\}$ has positive topological entropy with respect to $T$.

We have the following characterization of entropy tuples by independence sets,
which was firstly proved in \cite[Theorem 7.3]{HY06} for $\mathbb{Z}$-actions and
was generalized in \cite{KL07} for countable discrete amenable group actions.
\begin{thm}\label{thm:entropy-tuple-independence-sets}
Let $(X,T)$ be a dynamical system and $n\geq 2$.
Then a non-diagonal tuple $(x_1,x_2,\dotsc,x_n)\in X^n$ is a topological entropy $n$-tuple
if and only if for any neighborhood $U_i$ of $x_i$,
$(U_1, U_2, \dotsc,U_n)$ has an independence set with positive density.
\end{thm}

Let $K \subset X$ be a non-empty set and $\U$ be an open cover of $X$.
We say that  $\mathcal U$ is {\it admissible}
with respect to $K$ if for any $U\in\U$, $K$ is not contained in the closure of $U$.
A closed subset $K$ of $X$ with at
least two points is called an {\it entropy set}
if for any open cover $\U$ of $X,$ admissible with respect to $K$,
the topological entropy of $\U$ with respect to $T$ is positive.
Denote by $E_s(X,T)$ the collection of all entropy sets of $(X,T)$
and  by $H(X,T)$ the closure of $ E_s(X,T)$ in $2^X$.
An {\it entropy point} is a point $x\in X$ such that the singleton $\{x\}$ is in $H(X,T)$.
Denote by $E_1(X,T)$ the set of all entropy points of $(X,T)$.
It is not hard to see that $H(X,T)=E_s(X,T)\cup\{\{x\}\colon x\in E_1(X,T)\}$.
We also have the following observation.
\begin{lem}
Let $(X,T)$ be a dynamical system.
If $K\in 2^X$ contains at least two points, then
$K$ is an entropy set if and only if for any $k_1,k_2,\dotsc,k_n\in K$ with $|\{k_1,k_2,\dotsc,k_n\}|\geq 2$,
one has $(k_1,k_2,\dotsc,k_n)$ is a topological entropy $n$-tuple.
\end{lem}

In this section, we focus on proving Theorem B.
To do this, we need more results in~\cite{KL07} related to the independence sets.
\begin{lem}[\mbox{\cite[Lemma 3.8]{KL07}}]\label{lem3.33} Let $(X,T)$ be a dynamical system.
Let $(A_1, \dotsc , A_k )$ be a $k$-tuple of subsets of $X$ which has an independence set of positive density.
Suppose that $A_1 = A_{1,1} \cup A_{1,2}$. Then at least one
of the tuples $(A_{1,1}, \dotsc, A_k )$ and $(A_{1,2}, \dotsc , A_k )$ has an independence set of positive
density.
\end{lem}

For any tuple $\bm{A}=(A_1,\dotsc,A_k)$ of subsets of $X$,
let $\mathcal{P}_{\bm{A}}$ be the set of all independence sets for $\bm{A}$.
By taking indicator functions, the subsets of $\mathbb{Z}_+$ can be viewed as the elements of
$\Sigma_2:=\{0,1\}^{\mathbb{Z}_+}$.
The shift map $\sigma$ on $\Sigma_2$ is induced by the subtraction by $1$, that is
for every $F\subset \mathbb{Z}_+$, $\sigma(F)=F-1:=\{f-1\ge 0: f\in F\}$.
Clearly, $\mathcal{P}_{\bm{A}}$ is closed and shift invariant.

We say a closed shift invariant subset $\mathcal{P}\subset\Sigma_2$ has {\it positive density}
if it has an element with positive density.
Then a tuple $\bm{A}$ has an independence set of positive density
exactly when $\mathcal{P}_{\bm{A}}$ has positive density.
We also say $\mathcal{P}$ is {\it hereditary} if
any subset of any element in $\mathcal{P}$ is an element in $\mathcal{P}$.
It is obvious that $\mathcal{P}_{\bm{A}}$ is hereditary.
For any subset $\mathcal{P}\subset\Sigma_2$,
we say that a finite subset $J\subset \mathbb{Z}_+$ has {\it positive density with respect to $\mathcal{P}$}
if there exists a $K\subset \mathbb{Z}_+$ with positive density such that $(K-K)\cap(J-J)=\{0\}$
and $K+J\in \mathcal{P}$.
We say that a  subset $J\subset \mathbb{Z}_+$ has {\it positive density with respect to $\mathcal{P}$}
if every finite subset of $J$ has positive density with respect to $\mathcal{P}$.

\begin{lem}[\mbox{\cite[Lemma 3.17]{KL07}}]\label{important}
If $\mathcal{P}$ is a hereditary closed shift invariant subset of $\Sigma_2$ with positive density,
then there exists a $J\subset\mathbb Z_+$ with positive density
which also has positive density with respect to $\mathcal{P}$.
\end{lem}

We need to mention that the Lemma 3.17 in \cite{KL07} was shown for the case $\{0,1\}^{\mathbb{Z}}$, but its proof is also valid for the case $\Sigma_2:=\{0,1\}^{\mathbb{Z}_+}$.
Now, we start to prove Theorem B. Note that in the proof we need to use the Szemeredi's Theorem,
which states that any subset of $\Z$ with positive density contains
arithmetic progressions of arbitrary finite length.

\begin{proof}[Proof of Theorem B]
By Lemma~\ref{lem:Delta-WM-G-delta}, we know that the collection of $\Delta$-weakly mixing sets
is a $G_\delta$ subset of $2^X$.
So it is enough to prove that this collection is dense in the collection of entropy sets.
Now fix an entropy set $E$ and $\varepsilon>0$.
Since $E$ is compact and contains at least two points, there exists a finite subset $\{x_1,x_2,\dotsc,x_n\}$
of $E$ with $n\geq 2$ such that $\rho_H(E,\{x_1,x_2,\dotsc,x_n\})<\varepsilon/2$.
Choose pairwise disjoint closed neighborhoods
$A_i$ of $x_i$ with $\diam(A_i)<\varepsilon/2$ for $i=1,2,\dotsc,n$.
Let $\mathcal{E}_1=\{1,2,\dotsc,n\}$, $\mathcal{E}_2=\mathcal{E}_1\times \mathcal{E}_1=\{1,2,\dotsc,n\}^2$,
$\mathcal{E}_3=\mathcal{E}_2\times \mathcal{E}_2 \times  \mathcal{E}_2=\{1,2,\dotsc,n\}^6$
and
\[\mathcal{E}_{k}=\mathcal{E}_{k-1}\times \mathcal{E}_{k-1} \times \dotsb\times \mathcal{E}_{k-1}\ (k\text{ times})
=\{1,2,\dotsc,n\}^{k!}.\]

We shall construct, via induction on $k$, non-empty closed subsets $A_\sigma$ for $\sigma\in \mathcal{E}_k$
with the following properties:
\begin{enumerate}
  \item when $k=1$, $A_\sigma=A_{\sigma(1)}$ for $\sigma\in \mathcal{E}_1$,
  \item when $k\geq 2$, there exists $n_k\in\N$ such that
  for any $\sigma=\sigma_1\sigma_2\dotsb\sigma_k\in\mathcal{E}_k$, where $\sigma_i\in\mathcal{E}_{k-1}$ for
  $i=1,2,\dotsc,k$, one has
  \[A_\sigma\subset A_{\sigma_1}\cap T^{-n_k} A_{\sigma_2}\cap \dotsb\cap T^{-(k-1)n_k} A_{\sigma_k},\]
  \item when $k \geq 2, \diam(A_\sigma) \leq 2^{-k}$ for all $\sigma\in\mathcal{E}_k$,
  \item for every $k\geq 1$,
  the collection $\{A_\sigma\colon \sigma\in \mathcal{E}_k\}$, ordered into a tuple,
  has an independence set of positive density.
\end{enumerate}

Suppose that we have constructed all the sets $A_\sigma$.
Let
\[A=\bigcap\limits_{k=1}^\infty\bigcup\limits_{\sigma\in\mathcal E_k} A_{\sigma}.\]
Note that $A_\sigma$ for all $\sigma$ in a given $\mathcal{E}_k$ are pairwise disjoint
because of property (2). Then $A$ is a Cantor set.
It clear that $A\subset \bigcup_{i=1}^nA_i$ and $A\cap A_i\neq\emptyset$ for $i=1,2\dotsc,n$.
Then $\rho_H(A,\{x_1,x_2,\dotsc,x_n\})<\varepsilon/2$ and  $\rho_H(E,A)<\varepsilon$.
By (4) and Theorem~\ref{thm:entropy-tuple-independence-sets} we know that $A$ is an entropy set.
We use (2) and Proposition~\ref{prop:Delta-WM-independent-set} to show that $A$ is $\Delta$-weakly mixing.
Fix $d\geq 2$ and non-empty open subsets $U_1,U_2,\dotsc,U_d$ of $X$ intersecting $A$.
Pick a sufficiently large integer $k\geq d$ such that for every $i=1,2,\dotsc,d$
there exists $\alpha_i\in \mathcal{E}_k$ such that $A_{\sigma_i}\subset U_i$ and
all those $\alpha_i$ are pairwise distinct.
By (2) there exists $n_{k+1}\in\N$ as required.
Then $\{0,n_{k+1},\dotsc,(d-1)n_{k+1}\}$ is an independent set for $(U_1,U_2,\dotsc,U_d)$
with respect to $A$. Thus $A$ is $\Delta$-weakly mixing.

Now we construct the sets $A_\sigma$. Define $A_\sigma$ for $\sigma\in\mathcal{E}_1$ according to the property (1).
By the assumption the property (4) is satisfied for $k=1$.
Assume that we have constructed $A_\sigma$ for all $\sigma\in \mathcal{E}_j$ and $j=1,2,\dotsc,k$
with properties (1)-(4).
Set $\bm{A}_k$ to be $\{A_\sigma\colon \sigma\in \mathcal{E}_k\}$ ordered into a tuple.

By Lemma~\ref{important},
there exists an $H\subset \mathbb{Z}_+$ with positive density which also has positive density with
respect to $\mathcal{P}_{\bm{A}_k}$.
By the Szemeredi's Theorem, there exist $ m_k,n_k\in\N$ such that $\{m_k,m_k+n_k,\dotsc,m_k+(k-1)n_k\}\subset H$.
Note that $\{m_k,m_k+n_k,\dotsc,m_k+(k-1)n_k\}$ has positive density with
respect to $\mathcal{P}_{\bm{A}_k}$, then so does $\{0,n_k,\dotsc,(k-1)n_k\}$.

Let $J=\{0,n_k,\dotsc,(k-1)n_k\}$.
There exists a $K\subset \mathbb{Z}_+$ with positive density such that
$$(K-K)\cap(J-J)=\{0\} \text{ and } K+J\in \mathcal{P}_{\bm{A}_k}.$$
Then $K$ is an independence set of
the collection
\[\{ A_{\sigma_1}\cap T^{-n_k}A_{\sigma_2}\cap \dotsb\cap T^{-(k-1)n_k} A_{\sigma_k}\colon
\sigma_i\in \mathcal{E}_k\ \text{for}\ i=1,2,\dotsc,k\} \]
ordered into a tuple.

It follows from Lemma \ref{lem3.33} that
for every $\sigma=\sigma_1\sigma_2\dotsb\sigma_k\in\mathcal{E}_{k+1}$
there exists a non-empty closed subset
$A_{\sigma}\subset A_{\sigma_1}\cap T^{-n_k}A_{\sigma_2}\cap \dotsb\cap T^{-(k-1)n_k} A_{\sigma_k}$
such that $\diam(A_{\sigma})<1/(k+1)$ and
the collection
$\{A_{\sigma}\colon \sigma\in\mathcal{E}_{k+1}\}$,
ordered into a tuple, has independence set of positive density.
Then properties (2)-(4) hold for $k+1$.
This ends the induction procedure and hence the proof of the theorem.
\end{proof}

\section{Dynamical systems admitting a non-measurable distal ergodic measure}
In this section, we aim to prove Theorem C.
We refer the reader to \cite{F81,G03} for basic notions in ergodic theory,
especially for the decomposition of the measure over a factor.
In \cite{BGKM02} the authors used the Furstenberg-Zimmer structure theorem of an ergodic system showing that
a dynamical systems admitting a non-measurable distal ergodic measure is Li-Yorke chaotic.
Since in our discussion the iterates of the map is involved, we can not simply use the structure theorem.
However, we need some results of Furstenberg in the proof of Szemeredi's Theorem.
Firstly, we state a proposition related to the measurable weak mixing extension in \cite{F81}.

\begin{prop}[\mbox{\cite[Proposition 7.8]{F81}}] Let $\Gamma$ be an ablelian group.
If $(X,\mathcal{B},\mu,\Gamma)\to (Y,\mathcal{D},\nu,\Gamma)$ is weakly mixing relative to $\Gamma$
and $T_1,T_2,\dotsc,T_\ell$ are distinct elements of $\Gamma$,
then for any functions $f_1,f_2,\dotsc,f_\ell\in L^\infty(X)$,
\[\frac{1}{N}\sum_{n=1}^N\int \left\{ E\left(\prod_{i=1}^\ell T^n_if_i|Y\right)-\prod_{i=1}^\ell T^n_i E(f_i|Y)\right\}^2d\nu\to 0\]
as $N\to\infty$.
\end{prop}

The following proposition is a direct consequence of the above one.

\begin{prop}[\mbox{\cite[Lemma 7.9]{F81}}]\label{prop:WM-F81-prop78} Let $\Gamma$ be an ablelian group.
If $(X,\mathcal{B},\mu,\Gamma)\to (Y,\mathcal{D},\nu,\Gamma)$ is weakly mixing relative to $\Gamma$
and $T_1,T_2,\dotsc,T_\ell$ are distinct elements of $\Gamma$, and let $f\in L^\infty(X)$.
Let $\varepsilon,\delta>0$.
For each $y$, let $\psi(y)=\int f d\mu_y$, where $\mu=\int \mu_yd\nu$ is the decomposition of the measure $\mu$ over $\nu$.
Then the set of
$n\geq 0$ for which
\[\mu\left\{y\in Y: \left|\int \prod T_i^n fd\mu_y-\prod \psi(T_i^ny)\right|>\varepsilon\right\}>\delta\]
has density $0$.
\end{prop}

We restate Proposition \ref{prop:WM-F81-prop78} in our purpose as following.

\begin{prop} \label{prop:WM-density-one}
Suppose that $(X,\mathcal{B},\mu,T)$ is ergodic and
$\pi:(X,\mathcal{B},\mu,T)\to (Y,\mathcal{D},\nu,S)$ is weakly mixing.
Let $\mu=\int \mu_yd\nu$ be the decomposition of the measure $\mu$ over $\nu$.
Let $r\geq 2$ and $A_1,A_2,\dotsc,A_r\in \mathcal{B}$.
Then for every $\varepsilon,\delta>0$, there exists a subset $F\subset \N$ with density one
 such that for any $n\in F$
\begin{align*}
  \nu(\{y\in Y\colon &\big|\mu_y(A_{1}\cap T^{-n}A_{2}\cap\dotsb\cap T^{-(r-1)n}A_{r})\\
&\quad - \mu_y(A_{1})\mu_{S^ny}(A_{2})\dotsb\mu_{S^{(r-1)n}y}(A_{r})\big|< \varepsilon\})\geq 1-\delta.
\end{align*}
\end{prop}

The following multiple ergodic theorem is also needed in our proof.
\begin{thm}[\mbox{\cite[Theorem7.13]{F81}}] \label{thm:F81-SZ}
If $T_1,T_2,\dotsc,T_\ell$ are commuting invertible measure preserving transformations of a
probability measure space $(X,\mathcal{B},\mu)$ and $A\in\mathcal{B}$ with $\mu(A)>0$, then
\[\liminf_{N\to\infty} \frac{1}{N}\sum_{n=1}^N
\mu(T_1^{-n}A\cap T_2^{-n}A\cap \dotsb\cap T_\ell^{-n}A)>0.\]
\end{thm}

We have following result on the weakly mixing extension.

\begin{prop} \label{prop:wm-extension}
Suppose that $(X,\mathcal{B},\mu,T)$ is ergodic and
$\pi:(X,\mathcal{B},\mu,T)\to (Y,\mathcal{D},\nu,S)$ is weakly mixing.
Let $\mu=\int \mu_yd\nu$ be the decomposition of the measure $\mu$ over $\nu$.
For $A_1,A_2,\dotsc,A_k\in\mathcal{B}$, if
$\Omega=\{y\in Y: \mu_y(A_i)>0, i=1,2,\dotsc,k\}$ has positive $\nu$-measure,
then for every $r\in\N$,
there exist $n\in \N$ and $c'>0$ such that
\[\Omega'=\{y\in Y: \mu_y(A_s)>c', s\in\{1,\dotsc,k\}^r\}\]
has positive $\nu$-measure,
where
\[A_s=A_{s(1)}\cap T^{-n}A_{s(2)}\cap\dotsb\cap T^{-(r-1)n}A_{s(r)}\]
for $s\in\{1,\dotsc,k\}^r$.
\end{prop}

\begin{proof}
Since $\Omega=\bigcup\limits_{n=1}^\infty \{y\in Y: \mu_y(A_i)>1/n, i=1,2,\dotsc,k\}$
has positive $\nu$-measure, there exists $c>0$ such that
\begin{align*}
\Omega_c:=\{y\in Y: \mu_y(A_i)>c, i=1,2,\dotsc,k\}
\end{align*}
has positive $\nu$-measure.

Fix $r\in\N$. First, applying Theorem~\ref{thm:F81-SZ} to $id, S, S^2, \dotsc, S^{r-1}$ and $\Omega_c$,
there exists $\lambda>0$ such that
\[\liminf_{N\to\infty} \frac{1}{N}\sum_{n=0}^{N-1} \nu(\Omega_c\cap S^{-n}(\Omega_c)\cap \dotsb\cap S^{-(r-1)n}(\Omega_c))>\lambda>0.\]
Then there exists a subset $H$ of $\N$ with positive lower density such that for every $n\in H$,
\[\nu(\Omega_c\cap S^{-n}(\Omega_c)\cap \dotsb\cap S^{-(r-1)n}(\Omega_c))>\lambda/2.\]

Fix $\varepsilon=c^r/10$ and choose $0<\delta <\lambda/2$.
For any $s \in\{1,2,\dotsc,k\}^r$, by Proposition \ref{prop:WM-density-one},
there exists a subset $F_s\subset \N$ with density one
 such that for any $n\in F_s$
\begin{align*}
  \nu(\{y\in Y\colon &|\mu_y(A_{s(1)}\cap T^{-n}A_{s(2)}\cap\dotsb\cap T^{-(r-1)n}A_{s(r)})\\
&\quad - \mu_y(A_{s(1)})\mu_{S^ny}(A_{s(2)})\dotsb\mu_{S^{(r-1)n}y}(A_{s(r)})|< \frac{c^r}{10}\})\geq 1-\frac{\delta}{k^r}.
\end{align*}
Let $F=\bigcap_s F_s$.
Then $F$ also has density one.
For every $n\in F$, one has
\begin{align*}
\nu(\{y\in Y\colon &|\mu_y(A_{s(1)}\cap T^{-n}A_{s(2)}\cap\dotsb\cap T^{-(r-1)n}A_{s(r)})\\
&\quad - \mu_y(A_{s(1)})\mu_{S^ny}(A_{s(2)})\dotsb\mu_{S^{(r-1)n}y}(A_{s(r)})|<\frac{c^r}{10},
s\in\{1,2,\dotsc,k\}^r\}) \geq 1-\delta.
\end{align*}
It is clear that $F\cap H\neq\emptyset$.
Pick $n\in F\cap H$. Let
\begin{align*}
  \Omega'=&\{y\in Y\colon |\mu_y(A_{s(1)}\cap T^{-n}A_{s(2)}\cap\dotsb\cap T^{-(r-1)n}A_{s(r)})\\
&\qquad\qquad- \mu_y(A_{s(1)})\mu_{S^ny}(A_{s(2)})\dotsb\mu_{S^{(r-1)n}y}(A_{s(r)})|<\frac{c^r}{10}, s\in\{1,2,\dotsc,k\}^r\}\\
&\cap\Omega_c\cap S^{-n}(\Omega_c)\cap \dotsb\cap S^{-(r-1)n}(\Omega_c).
\end{align*}
and $c'=9c^r/10$.
Then $\nu(\Omega')>(1-\delta)+\lambda/2-1=\lambda/2-\delta>0$.

Fix $y\in \Omega'$. Then
\[
y\in\Omega_c, S^ny\in\Omega_c,\dotsc, S^{(r-1)n}y\in \Omega_c.
\]
For any $s\in\{1,2,\dotsc,k\}^r$,
one has
\[\mu_y(A_{s(1)})\mu_{S^ny}(A_{s(2)})\dotsb\mu_{S^{(r-1)n}y}(A_{s(r)})>c^r.\]
Then
\begin{align*}
  &\hskip0.5cm \mu_y(A_{s(1)}\cap T^{-n}A_{s(2)}\cap\dotsb\cap T^{-(r-1)n}A_{s(r)})\\
  &\geq\mu_y(A_{s(1)})\mu_{S^ny}(A_{s(2)})\dotsb\mu_{S^{(r-1)n}y}(A_{s(r)})-\frac{c^r}{10}\\
  &> c^r-\frac{c^r}{10}=c'.
\end{align*}
This ends the proof.
\end{proof}

The following technical lemma is also needed.
\begin{lem} \label{lem:to-small-sets}
Suppose that $(X,\mathcal{B},\mu,T)$ is ergodic and
$\pi:(X,\mathcal{B},\mu,T)\to (Y,\mathcal{D},\nu,S)$ is weakly mixing.
Let $\mu=\int \mu_yd\nu$ be the decomposition of the measure $\mu$ over $\nu$.
Suppose that $A_1,A_2,\dotsc,A_k\in\mathcal{B}$ and
\[\Omega=\{y\in Y: \mu_y(A_i)>c, i=1,2,\dotsc,k\}\]
has positive $\nu$-measure.
If $A_1=A_{1,1}\cup A_{1,2}$. Then there exist $j\in\{1,2\}$ and $c'>0$ such that
\[\Omega'=\{y\in Y\colon \mu_y(A_{1,j})>c', \mu_y(A_i)>c', i=2,\dotsc,k\}\]
has positive $\nu$-measure.
\end{lem}
\begin{proof}
First, we have
\begin{align*}
&\hskip0.6cm \int \left[\mu_y(A_{1,1})\prod_{i=2}^k
   \mu_y(A_i)\right] d\nu(y)+\int \left[\mu_y(A_{1,2})\prod_{i=2}^k
   \mu_y(A_i)\right] d\nu(y)  \\
&=\int \left[(\mu_y(A_{1,1})+\mu_y(A_{1,2}))\prod_{i=2}^k
   \mu_y(A_i)\right] d\nu(y)\\
&\geq
\int \prod_{i=1}^k \mu_y(A_i) d\nu(y)\geq c^k\nu(\Omega)>0.
\end{align*}
Without loss of generality, assume that
\[\int \left[\mu_y(A_{1,1})\prod_{i=2}^k \mu_y(A_i)\right] d\nu(y)>0\]
Then there exists a subset $\Omega_0$ of $Y$ with positive $\nu$-measure such that
for every  $y\in \Omega_0$,
\[\mu_y(A_{1,1})>0, \mu_y(A_{i})>0, i=2,\dotsc,k.\]
For every $n\in\N$,
put
\[\Omega_n=\{y\in \Omega_0\colon \mu_y(A_{1,1})>1/n, \mu_y(A_{i})>1/n, i=2,\dotsc,k\}.\]
Clearly,
\[\bigcup_{n\in\N}\Omega_n=\Omega_0.\]
 Then there exists $n_0\in\N$ such that $\nu(\Omega_{n_0})>0$.
Set $c=1/n_0$ and $\Omega'=\Omega_{n_0}$, which are as required.
\end{proof}

After these preparations, we are ready to prove Theorem C. In fact we show a litter more, namely
the $\Delta$-weak mixing set exists in  some fibre of $\pi$, where $\pi$
is the factor map to the maximal measurable distal factor.
\begin{proof}[Proof of Theorem C]
By the Furstenberg-Zimmer theorem \cite{F81,Z76,Z76-2}, the ergodic system
$(X,\mathcal{B},\mu,T)$ admits a maximal measurable distal factor
$\pi:(X,\mathcal{B},\mu,T)\to (Y,\mathcal{D},\nu,S)$
with the factor map $\pi$ being a weakly mixing extension.
Let $\mu=\int \mu_yd\nu$ be the decomposition of the measure $\mu$ over $\nu$.
Let $Y_0:=\{y\in Y:\supp \mu_y\subset \supp \mu\}$. Then $\nu(Y)=1$.
For $n\geq 2$, put $\lambda_n=\mu \times_Y\times_Y\dotsb\times_Y \mu$ ($n$-times).
Fix a non-diagonal tuple $(x_1,x_2,\dotsc,x_n)\in \mathrm{\supp}(\lambda_n)$.
Without loss of generality, assume that $x_1,x_2,\dotsc,x_n$ are pairwise different.
Fix a closed neighbourhood $V_i$ of $x_i$ such that $V_i\cap V_j=\emptyset$ for $i\neq j$.
We will show that there is a $\Delta$-weakly mixing set
$A\in \langle V_1,\dotsc, V_n\rangle$ in the fiber of $\pi$.

As $\lambda_n(V_1\times V_2\times\dotsb\times V_n)>0$, there exists $c_1>0$
such that the set
\[\{y\in Y_0\colon \mu_y(V_i)>c_1, i=1,2,\dotsc,n\}\]
has positive $\nu$-measure.
Since $(Y,\mathcal{D},\nu,S)$ is a Lebesgue space, we fix a separated metric on $Y$ such that
every open set is $\mathcal{D}$-measurable.
Then we pick $C_1\subset\{y\in Y_0: \mu_y(V_i)>c_1, i=1,2,\dotsc,n \}$ with $\diam(C_1)<1$ and $\nu(C_1)>0.$
Set $A_i=V_i\cap\supp(\mu)\cap \pi^{-1}(C_1)$ for $i=1,2,\dotsc,n$.
Then the set
\[\{y\in Y_0\colon \mu_y(A_i)>c_1, i=1,2,\dotsc,n\}\]
has positive $\nu$-measure.

Let $\mathcal{E}_1=\{1,2,\dotsc,n\}$,
$\mathcal{E}_2=\mathcal{E}_1\times \mathcal{E}_1=\{1,2,\dotsc,n\}^2$,
and
\[\mathcal{E}_{k}=\mathcal{E}_{k-1}\times \mathcal{E}_{k-1} \times \dotsb\times \mathcal{E}_{k-1}\ (k\text{ times})
=\{1,2,\dotsc,n\}^{k!}.\]
We shall construct, via induction on $k$, non-empty closed subsets $A_\sigma$ for $\sigma\in \mathcal{E}_k$
with the following properties:
\begin{enumerate}
  \item when $k=1$, $A_\sigma=A_{\sigma(1)}$ for $\sigma\in \mathcal{E}_1$.
  \item when $k\geq 2$, there exist $n_k\in\N$, $C_k\subset Y$ such that $\nu(C_k)>0$ and
  for any $\sigma=\sigma_1\sigma_2\dotsb\sigma_k\in\mathcal{E}_k$, where $\sigma_i\in\mathcal{E}_{k-1}$ for
  $i=1,2,\dotsc,k$, one has
  \[A_\sigma\subset A_{\sigma_1}\cap T^{-n_k} A_{\sigma_2}\cap \dotsb\cap T^{-(k-1)n_k} A_{\sigma_k}\cap \supp(\mu)\cap\pi^{-1}(C_k).\]
  \item when $k \geq 2, \diam(A_\sigma) \leq 2^{-k}$ for all $\sigma\in\mathcal{E}_k$ and $\diam C_k<1/k$.
  \item for every $k\geq 1$, there exists $c_k>0$ such that
  the set
  \[\{y\in Y\colon \mu_y(A_\sigma)>c_k, \sigma\in\mathcal{E}_k\}\]
  has positive $\nu$-measure.
\end{enumerate}

Suppose that we have constructed all the sets $A_\sigma$.
Let
\[A=\bigcap_{k=1}^\infty\bigcup_{\sigma\in\mathcal{E}_k} A_\sigma.\]
Similar to the proof of Theorem B, we know that $A$ is $\Delta$-weakly mixing.
Since $A\subset \bigcap\limits_{k\in\mathbb N}\pi^{-1}(C_k)$,
we get that $\pi(A)\subset C_k$ for each $k\in\N$.
Since $\diam(C_k)\to0$, $k\to\infty$,
we know that $\diam(\pi(A))=0$ which implies that there exists $y_0\in Y$ such that $A\subset \pi^{-1}(y_0)$.

We now construct the $A_\sigma$.
Define $A_\sigma$ for $\sigma\in \mathcal{E}_1$  according to property (1).
By the construction of $A_i$, property (4) is satisfied form $k=1$.

Assume that we have constructed $A_\sigma$ for
all $\sigma\in\mathcal{E}_j$ and $j = 1,\dotsc,k$ with the above properties.
Let $n_{k+1}\in\N$ and $c'>0$ be in Proposition \ref{prop:wm-extension} by applying the collection
$\{A_\sigma\colon \sigma\in\mathcal{E}_k\}$ and $r=k+1$.
For any $\sigma=\sigma_1\sigma_2\dotsb\sigma_{k+1}\in\mathcal{E}_{k+1}$,
where $\sigma_i\in\mathcal{E}_{k}$ for $i=1,2,\dotsc,k+1$,
put
\[B_\sigma= A_{\sigma_1}\cap T^{-n_k} A_{\sigma_2}\cap \dotsb\cap T^{-kn_k} A_{\sigma_{k+1}}.\]
Then the set
\[\{y\in Y\colon \mu_y(B_\sigma)>c', \sigma\in \mathcal{E}_{k+1}\}\]
has positive $\nu$-measure.
Choose $C_k\subset\{y\in Y_0:\mu_y(B_\sigma)>c', \sigma\in \mathcal{E}_{k+1}\}$
such that $\diam(C_k)<1/k,\nu(C_k)>0$ and
\[\{y\in Y\colon \mu_y(B_\sigma\cap\supp(\mu)\cap\pi^{-1}(C_k))>c', \sigma\in \mathcal{E}_{k+1}\}\]
has positive $\nu$-measure.
By Lemma \ref{lem:to-small-sets}, there exists $c_{k+1}>0$ and
closed subsets $A_\sigma\subset B_\sigma\cap \supp(\mu)\cap\pi^{-1}(C_k)$
with $\diam(A_\sigma)<2^{-(k+1)}$ for all $\sigma\in\mathcal{E}_{k+1}$
such that
the set
\[\{y\in Y\colon \mu_y(A_\sigma)>c_{k+1}, \sigma\in \mathcal{E}_{k+1}\}\]
has positive $\nu$-measure.
Consequently, properties (2-4) are also satisfied for $k+1$.
This completes the induction procedure and hence the proof of the theorem.
\end{proof}

\begin{rem}
In the proof of Theorem C, we show that there are many $\Delta$-weak mixing sets in the fiber of $\pi$.
In fact, we conjecture that for a.e.\ $y\in Y$, $\pi^{-1}(y)$ contains a $\Delta$-weak mixing set.
\end{rem}

\section{Non-classical Li-Yorke chaos}
In this section, we study the non-classical Li-Yorke chaos.
We will answer affirmatively three open questions stated in \cite{M12} and generalize some results
obtained in \cite{HY02, AGHSY} in this section. To do this we start from some notions.

Following the idea in \cite{LY75}, we usually define the Li-Yorke chaos as follows.
Let $(X,T)$ be a dynamical system.
A pair $(x,y) \in X\times X$ is called {\it scrambled} if
\[\liminf_{n\to\infty}\rho(T^nx,T^ny)=0\text{ and }\limsup_{n\to\infty}\rho(T^nx,T^ny)>0.\]
A subset $C$ of $X$ is called {\it scrambled} if any
two distinct points $x,y\in C$ form a scrambled pair.
The dynamical system $(X,T)$ is called {\it Li-Yorke chaotic} if there is an uncountable scrambled set in $X$.
In 2002  Huang and Ye \cite{HY02} showed that Devaney chaos implies Li-Yorke one
by proving that a non-periodic transitive system with a periodic point is Li-Yorke chaotic.
Also in 2002, Blanchard, Glasner, Kolyada and Maass \cite{BGKM02} proved that
positive topological entropy also implies Li-Yorke chaos.

Among other things, Moothathu \cite{M12} introduced a concept of non-classical Li-Yorke chaos.
For two positive integers $r,s\in\N$, we say that a pair $(x,y) \in X\times X$ is {\it $(r,s)$-scrambled}
if
\[\liminf_{n\to\infty} \rho(T^{rn}x,T^{sn}y)=0\text{ and }\limsup_{n\to\infty} \rho(T^{rn}x,T^{sn}y)>0.\]
Similarly, we can define {\it $(r,s)$-scrambled sets} and {\it $(r,s)$-Li-Yorke chaos}.
The classical Li-Yorke chaos is just $(1,1)$-Li-Yorke chaos.
As we know positive topological entropy implies Li-Yorke chaos,
Moothathu proposed the following natural question in \cite{M12}.
\begin{ques}\label{q-1}
Does positive topological entropy imply the existence of plenty of $(r,s)$-scrambled pairs
for two distinct integers $r,s\in\N$?
\end{ques}

Using the results developed in previous sections, we have a positive answer to this question.
In fact, we first show that the existence of $\Delta$-mixing sets implies $(r,s)$-Li-Yorke chaos for any $r,s\in\N$.

\begin{thm}\label{thm:Delta-WM-scrambed}
Let $(X,T)$ be a dynamical system.
If $E$ is a $\Delta$-weakly mixing set, then there exists a dense Mycielski subset $C$ of $E$ such that
$C$ is $(r,s)$-scrambled sets for any $r,s\in\N$.
\end{thm}
\begin{proof}
By Theorem~\ref{thm:Delta-weakly-mixing-set},
there exists a countable Cantor subsets $C_1\subset C_2\subset \dotsb$ of $E$
such that $C=\bigcup_{k=1}^\infty C_k$
is dense in $E$ and satisfies relevant conditions.
Then $C$ is a dense Mycielski subset of $E$.
It  is sufficient to show that $C$ is $(r,s)$-scrambled sets for any $r,s\in\N$.

Fix two positive integers $r,s\in\N$ and two distinct points $a,b\in E$.
For every two distinct points $x,y\in C$,
define $g_r\colon \{x,y\} \to E$ by $g_r(x)=a,g_r(y)=a$ and $g_s=g_r$.
By the conclusions of Theorem~\ref{thm:Delta-weakly-mixing-set},
there exists an increasing sequence $\{q_k\}_{k=1}^\infty$ of positive integers such that
  \[\lim_{i\to\infty}T^{r\cdot q_i}x=g_r(x)=a \text{ and }\lim_{i\to\infty}T^{s\cdot q_i}y=g_s(y)=a,\]
which implies that
\[\liminf_{n\to\infty} \rho(T^{rn}x,T^{sn}y)=0.\]
Now define $h_r\colon \{x,y\} \to E$ by $h_r(x)=a,h_r(y)=b$ and $h_s=h_r$.
By the conclusions of Theorem~\ref{thm:Delta-weakly-mixing-set} again,
there exists an increasing sequence $\{p_k\}_{k=1}^\infty$ of positive integers such that
  \[\lim_{i\to\infty}T^{r\cdot p_i}x=h_r(x)=a \text{ and }\lim_{i\to\infty}T^{s\cdot p_i}y=h_s(y)=b,\]
which implies that
\[\limsup_{n\to\infty} \rho(T^{rn}x,T^{sn}y)\geq \rho(a,b)>0.\]
Thus $(x,y)$ is $(r,s)$-scrambled, which ends the proof.
\end{proof}

Combining Theorems \ref{thm:positive-entopy-Delta-WM} and \ref{thm:Delta-WM-scrambed}, we have
\begin{thm}
Positive topological entropy implies   $(r,s)$-Li-Yorke chaos for any $r,s\in\N$.
\end{thm}

For convenience, we define the $(r,s)$-proximal relation and the $(r,s)$-asymptotic relation as
\begin{align*}
Prox(T^r,T^s)=\bigl\{(x,y)\in X^2\colon \liminf_{n\to\infty} \rho(T^{rn}x,T^{sn}y)=0\bigr\},\\
Asy(T^r,T^s)=\bigl\{(x,y)\in X^2\colon \lim_{n\to\infty} \rho(T^{rn}x,T^{sn}y)=0\bigr\}.
\end{align*}
It is clear that a pair $(x,y)\in X^2$ is $(r,s)$-scrambled if and only if
 $(x,y)\in Prox(T^r,T^s)\setminus Asy(T^r,T^s)$.
We also have the following observation.
\begin{lem}\label{lem:Prox-asy}
Let $(X,T)$ be a dynamical system and $k\in\N$.
Then $Prox(T,T)=Prox(T^k,T^k)$ and $Asy(T,T)=Asy(T^k,T^k)$.
\end{lem}

It was shown in \cite{HY02} that a non-periodic transitive system with a periodic point is
Li-Yorke chaotic. Moreover, it was proved in \cite{AGHSY} that if $(X,T)$ is transitive and there is
a subsystem $(Y,T)$ such that $(X\times Y, T\times T)$ is transitive, then $(X,T)$ is Li-Yorke chaotic.
We will extend the results to non-classical Li-Yorke chaos under the total transitivity assumption.
To this aim, first we need some results concerning the $(r,s)$-proximal relation and the $(r,s)$-asymptotic relation.

\begin{lem}[\mbox{\cite[Corollary 2.2]{HY02}}]\label{thm:asy}
If $X$ is infinite and $T$ is transitive, then $Asy(T,T)$ is of first category in $X^2$.
\end{lem}

\begin{lem}[\mbox{\cite[Theorem 4 and Theorem 6]{M12}}] \label{thm:asy-rs}
Let $(X,T)$ be a dynamical system.
\begin{enumerate}
\item If $T^s$ is transitive, then $Prox(T^r, T^{r+s})$ is a dense $G_\delta$ subset  in $X^2$ for each $r \in \N$.
\item If $X$ is infinite and $T$ is transitive,
then $Asy(T^r,T^{r+s})$ is of first category in $X^2$ for every $r,s\in\N$.
\end{enumerate}
\end{lem}

With the above lemmas we are able to show
\begin{thm}\label{thm:important}
Let $(X,T)$ be a non-trivial totally transitive system.
If $Prox(T,T)$ is dense in $X^2$, then
there exists a dense Mycielski set $S\subset X$
such that $S$ is $(r,s)$-scrambled for any $r,s\in\N$.
\end{thm}
\begin{proof}
As $T$ is totally transitive,
by Lemma~\ref{thm:asy-rs} $Prox(T^r,T^s)$ is a dense $G_\delta$ subset  in $X^2$
for any two distinct $r,s \in \N$.
By Lemma \ref{thm:asy-rs} $Asy(T^r,T^{s})$ is of first category in $X^2$ for any two distinct $r,s\in\N$.
Since $Prox(T,T)$ is  dense in $X^2$, by Lemma~\ref{lem:Prox-asy} $Prox(T^r,T^r)$ is
a dense $G_\delta$ subset  in $X^2$ for any $r\in \N$.
By Lemmas~\ref{lem:Prox-asy} and~\ref{thm:asy}
$Asy(T^r,T^{r})$ is of first category in $X^2$ for $r\in\N$.
Let
\[R=\bigcap_{r,s\in\N} Prox(T^r,T^s)\setminus \bigcup_{r,s\in\N} Asy(T^r,T^s).\]
Then $R$ is residual in $X\times X$.
By the Mycielski Theorem, there exists a  dense Mycielski set $S\subset X$
such that for any two distinct points $x,y\in S$, the pair $(x,y)\in R$.
It is clear that $S$ is $(r,s)$-scrambled for any $r,s\in\N$.
\end{proof}

Recall that a dynamical system is {\it scattering}, if its Cartesian product with any minimal system is transitive.
It is clear that a scattering system is totally transitive.  As applications of Theorem \ref{thm:important} we have
\begin{cor}
Let $(X,T)$ be a dynamical system.
If it satisfies one of the following conditions,
\begin{enumerate}
  \item $T$  is totally transitive and has a periodic point,
  \item $T$ is scattering,
  \item $T$ is weakly mixing,
\end{enumerate}
then $(X,T)$ is non-classic Li-Yorke chaotic.
\end{cor}
\begin{proof} Note that if $(X,T)$ satisfies one of the conditions, then $Prox(T,T)$ is dense in $X^2$
(see Theorem 4.1 and Theorem 4.3 in \cite{HY02}).
So Theorem~\ref{thm:important} applies.
\end{proof}

\begin{cor}
Let $(X,T)$ be a totally transitive system.
If there exists a subsystem $(Y,T)$ such that $(X\times Y,T\times T)$ is transitive,
then $(X,T)$ is non-classic Li-Yorke chaotic.
\end{cor}
\begin{proof} Note that if $(X,T)$ satisfies the assumption, then $Prox(T,T)$ is dense in $X^2$ (see the proof of Theorem 3.1 in \cite{AGHSY}).
So Theorem~\ref{thm:important} applies.
\end{proof}

We say that a dynamical system is {\it proximal} if every two points in $X$ form a proximal pair, that is
$Prox(T,T)=X^2$.
It is well known that $(X,T)$ is proximal if and only if
$(X,T)$ has only one minimal point which is also a fixed point (see for example, \cite{HY01, AK03}).
It was shown in \cite{M12} that $(X,T)$ is proximal if and only if $Prox(T^r,T^s)=X^2$ for every $r,s\in\N$.
If one only has $Prox(T^r, T^s) = X^2$ for some $r, s \in \N$ then
it is not difficult to show $(X,T)$ has only one minimal subsystem.
Moothathu asked the following question in \cite{M12}.

\begin{ques}\label{q-2}
Can we show that $T$ has only one minimal point if $Prox(T^r, T^s) = X^2$ for some $r, s \in \N$?
\end{ques}

We give a negative answer to this question. Recall
that an invertible dynamical system  $(X, T)$ is called {\it proximal orbit dense}, or a POD system, if
$(X, T)$ is
totally minimal; and whenever $x, y\in X$ with $x\neq y$, then for some $n\neq 0$, $T^ny$ is
proximal to $x$. An interesting property of POD is that (see \cite[Corollary 3.5]{KN}):

\begin{lem}\label{pod} Let $(X, T)$ be POD and $r,s\in \mathbb{N}$ with $r\neq s$. Then $(X\times X, T^r \times T^s)$ is minimal.
\end{lem}

A special class of POD is doubly minimal systems.

\begin{defn} An invertible dynamical system $(X, T)$ is said to be {\it doubly minimal} if for all $x\in X$ and
$y\not \in \{T^nx\}_{n\in \mathbb{Z}}$, the orbit $\{(T\times T)^n(x,y):n\in \mathbb{Z}\}$ of $(x, y)$ is dense in $X\times X$.
\end{defn}
The first example of a non-periodic doubly minimal system was constructed in
\cite{K} in the symbolic dynamics (in fact, any doubly minimal system is a subshift \cite{HY-cims}). Doubly minimal systems are natural in the sense
that any ergodic system with zero entropy has a uniquely ergodic model which is
doubly minimal \cite{W}. By Lemma \ref{pod}, we have the following result which serves a counterexample for Question \ref{q-2}:
\begin{prop}
Let $(X,T)$ be a POD system. Then $Prox(T^r, T^s) = X^2$ for $r,s\in \mathbb{N}$ with $r\neq s$.
\end{prop}

We call $(X,T)$ {\it completely scrambled} if  the whole space $X$ is a scrambled set.
In~\cite{HY01}, Huang and Ye constructed examples of compacta with
completely scrambled homeomorphisms.
Their examples, however, are not transitive.
Later in~\cite{HY02}, Huang and Ye showed that every almost equicontinuous
but not minimal system has a completely scrambled factor.
Recently, the authors in \cite{FHLO15} showed that
some weakly mixing completely scrambled systems which are proximal and uniform rigid.
Recall that a dynamical system is {\it uniform rigid} if there exists an increasing sequence
$\{n_k\}$ of positive integers such that
\[\lim_{k\to\infty}\sup_{x\in X}\rho(T^{n_k}x,x)=0,\]
that is $\{T^{n_k}\}_{k=1}^\infty$ converges uniformly to the identity map on $X$.
Moothathu proposed the following question in \cite{M12}.
\begin{ques}\label{q-3}
Is there a non-trivial dynamical system $(X,T)$ such that $(x,y)$
is $(r,s)$-scrambled  for any two distinct $x,y\in X$ and any two
distinct $r,s\in\N$?
\end{ques}
We give an affirmative answer to this question as the following theorem shows.
\begin{thm}
There are some dynamical systems $(X,T)$ such that $(x,y)$
is $(r,s)$-scrambled for any two distinct $x,y\in X$ and any $r,s\in\N$.
\end{thm}
\begin{proof}
Let $(X,T)$ be a dynamical system which is totally transitive, proximal and uniformly rigid.
It is shown in \cite[Proposition 5]{M12} 
that if $(X,T)$ is proximal then $Prox(T^r,T^s)=X^2$ for all $r,s\in \mathbb{N}$.
We also have $Asy(T^k,T^k)\setminus \Delta_X=\emptyset$ for $k\in\N$ by the uniform rigidity.
We want to show that $Asy(T^r,T^{r+s})\setminus \Delta=\emptyset$ for $r,s\in\N$.
Assume that $(x,y)\in Asy(T^r,T^{r+s})$.
First assume that $x$ is not a fixed point of $T$.
Since $x$ is recurrent for $T^r$, there exists a sequence $\{n_i\}$ of positive integers such that
\[x=\lim_{i\to\infty}T^{r\cdot n_i} (x)=\lim_{i\to\infty}T^{(r+s)\cdot n_i}(y).\]
By the continuity of $T$, we have
\[\lim_{i\to\infty}T^{r\cdot (n_i+1)} (x)=T^r(x)
\text{ and } \lim_{i\to\infty}T^{(r+s)\cdot (n_i+1)}(y)=T^{r+s}(x).\]
But
\[\lim_{i\to\infty}T^{r\cdot (n_i+1)}(x)= \lim_{i\to\infty}T^{(r+s)\cdot (n_i+1)}(y).\]
So $T^r(x)=T^{r+s}(x)$ and $x$ is periodic point.
But $x$ is not a fixed point, which contradicts to the proximality of $(X,T)$.

The same proof applies if $y$ is not a fixed point of $T$. If both of $x$ and $y$ are fixed points, then $x=y$ since $(x,y)\in Asy(T^r,T^{r+s})$.
This ends the proof.
\end{proof}

\bibliographystyle{amsplain}

\begin{thebibliography}{99}

\bibitem{A04} E. Akin, {\it Lectures on Cantor and Mycielski sets for dynamical systems}.
Chapel Hill Ergodic Theory Workshops,  21--79, Contemp. Math., 356, Amer. Math. Soc., Providence, RI, 2004.

\bibitem{AGHSY} E. Akin, E. Glasner, W. Huang, S. Shao and X. Ye,
{\it Sufficient conditions under which a transitive system is
chaotic}, Ergodic Theory Dynam. Systems  \textbf{30} (2010),  1277--1310.

\bibitem{AK03} E. Akin and S. Kolyada, {\it Li-Yorke sensitivity}, Nonlinearity \textbf{16} (2003), no.4, 1421--1433.

\bibitem{B92} F. Blanchard, {\it Fully positive topological entropy and topological mixing},
Symbolic dynamics and its applications (New Haven, CT, 1991), 95--105,
Contemp. Math., 135, Amer. Math. Soc., Providence, RI, 1992.

\bibitem{B93} F. Blanchard, {\it A disjointness theorem involving topological entropy},
Bull. Soc. Math. France \textbf{121} (1993), 465--478.

\bibitem{BGKM02}  F. Blanchard, E. Glasner, S.  Kolyada and A. Maass,  {\it On Li-Yorke pairs},
J. Reine  Ange. Math. \textbf{547} (2002), 51--68.

\bibitem{BH08} F. Blanchard and W. Huang, {\it Entropy sets, weakly mixing sets and entropy capacity},
Discrete Contin. Dyn. Syst.  \textbf{20} (2008), 275--311.

\bibitem{CLL14a} J. Chen, J. Li and J. L\"u,
{\it On multi-transitivity with respect to a vector}, Sci. China Math. \textbf{57} (2014) no. 8, 1639--1648.

\bibitem{CLL14} Z. Chen, J. Li and J. L\"u, {\it Point transitivity, $\Delta$-transitivity and multi-minimality}
 Ergodic Theory Dynam. Systems, published online 2014.

\bibitem{DYZ} D. Dou, X. Ye and G. Zhang, {\it Entropy
sequences and maximal entropy sets}, Nonlinearity \textbf{19} (2006), 53--74.

\bibitem{FHLO15} M. Fory\'s, W. Huang, J. Li, and P. Oprocha, {\it Invariant scrambled sets, uniform rigidity
and weak mixing}, to appear in Israel J. Math.

\bibitem{F81} H. Furstenberg, {\it Recurrence in Ergodic Theory and Combinatorial Number Theory},
Princeton Univ. Press, Princeton, NJ, 1981.

\bibitem{G94} E. Glasner, {\it Topological ergodic decompositions and applications to products of powers of
a minimal transformation}, J. Anal. Math. \textbf{64} (1994), 241--262.

\bibitem{G03} E. Glasner, {\it Ergodic theory via joinings}
Mathematical Surveys and Monographs, 101. American Mathematical Society, Providence, RI, 2003.

\bibitem{GW95} E. Glasner and B. Weiss, {\it Quasi-factors of zero entropy systems},
J. Amer. Math. Soc. \textbf{8} (1995), 665--686.


\bibitem{H06} W. Huang,
{\it Tame systems and scrambled pairs under an abelian group action},
Ergodic Theory Dynam. Systems \textbf{26} (2006), no.~5, 1549--1567.

\bibitem{HLY12}
W. Huang, H. Li, and X. Ye, {\it Family independence for
  topological and measurable dynamics}, Trans. Amer. Math. Soc. \textbf{364}
  (2012), no.~10, 5209--5242.

\bibitem{HY01} W. Huang and X.Ye, {\it Homeomorphisms with the whole compacta being scrambled sets},
Ergodic Theory Dynam. Systems \textbf{21} (2001), 77--91.

\bibitem{HY02} W. Huang and X. Ye, {\it Devaney's chaos or 2-scattering implies Li-Yorke's chaos},
 Topology Appl. \textbf{117}  (2002),  259--272.

\bibitem{HY06} W. Huang and X. Ye, {\it A local variational relation and applications},
Israel J. Math. \textbf{151} (2006), 237--279.

\bibitem{HY-cims} W. Huang and X. Ye, {\it A note on doubly minimality}, Comm. in Math. and Stat.,
{\bf 3}(2015), 57--61.

\bibitem{KL07} D. Kerr and H. Li, {\it Independence in topological and C*-dynamics},
Math. Ann. \textbf{338} (2007), 869--926.

\bibitem{KN} H.B. Keynes and D. Newton, {\it Real prime flows}, Trans. Amer. Math. Soc. \textbf{217} (1976),
 237--255.

\bibitem{K} Jonathan L. King, {\it A map with topological minimal self-joinings in the sense of del Junco}, Ergodic
Theory Dynam. Systems \textbf{10} (1990), no. 4, 745--761.

\bibitem{KLOY14} D.~Kwietniak, J. Li, P.~Oprocha and X. Ye,
{\it Multi-recurrence and van der Waerden systems}, preprint,
arXiv:1501.01491

\bibitem{KO12} D. Kwietniak and P. Oprocha, {\it On weak mixing, minimality and weak disjointness of all iterates},
 Ergodic Theory Dynam. Systems \textbf{32} (2012), 1661--1672.

\bibitem{L15} J. Li, {\it Localization of mixing property via Furstenberg families},
Discrete Contin. Dyn. Syst. \textbf{35} (2015), no. 2, 725--740.

\bibitem{LOZ15} J. Li, P. Oprocha, G. Zhang, {\it On recurrence over subsets and weak mixing},
to appear in Pac. J. Math.

\bibitem{LY75} T. Li and J. Yorke,
{\it Period three implies chaos}, Amer. Math. Monthly \textbf{82} (1975), 985--992.

\bibitem{O11} P. Oprocha, {\it Coherent lists and chaotic sets},
Discrete Contin. Dyn. Syst.  \textbf{31} (2011), no.~3, 797--825.

\bibitem{OZ12} P. Oprocha and G. Zhang, {\it On local aspects of topological weak mixing
in dimension one and beyond}, Studia Math.  \textbf{202} (2011), no.~3, 261--288.

\bibitem{M10}T.K.S.~Moothathu, {\it Diagonal points having dense orbit}, Colloq. Math. \textbf{120} (2010), 127--138.

\bibitem{M12}T.K.S.~Moothathu, {\it The speed with which an orbit approaches a limit point},
 J. Differ. Equ. Appl. \textbf{18} (2012), 1611--1621.

\bibitem{M64} J. Mycielski, {\it Independent sets in topological algebras}, Fund. Math. \textbf{55} (1964), 139--147.

\bibitem{N78} S. Nadler, {\it Hyperspaces of sets}, New York/Basel: Marcel Dekker, Inc., 1978.

\bibitem{S75} E. Szemeredi, {\it On sets of integers containing no $k$ elements in arithmetic progression},
Acta Arith. \textbf{27} (1975), 199--245.

\bibitem{W} B. Weiss, {\it Multiple recurrence and doubly minimal systems}, Topological dynamics and applications
(Minneapolis, MN, 1995), Contemp. Math., vol. 215, Amer. Math. Soc., Providence,
RI, 1998, pp. 189--196.

\bibitem{XY92} J. Xiong and Z. Yang, {\it Chaos caused by a toplogical mixing map}, Dynamical systems
and related topics (Nagoya, 1990), 550--572,  Adv. Ser. Dynam. Systems, 9, World Sci. Publ., River Edge, NJ, 1991.

\bibitem{Z76} R. J. Zimmer, {\it Extensions of ergodic group actions}, Ill. J. Math. \textbf{20} (1976), 373--409.
\bibitem{Z76-2} R. J. Zimmer, {\it Extensions of ergodic actions and generalized discrete spectrum},
Ill. J. Math. \textbf{20} (1976), 555--588.

\end{thebibliography}

\end{document}